\documentclass[letterpaper]{amsart}
\usepackage{fullpage}

\usepackage{enumerate}

\usepackage{mathrsfs}
\usepackage{amsfonts}
\usepackage{amsmath}
\usepackage{amssymb}
\usepackage{amsthm}
\usepackage{stmaryrd}
\usepackage{textcomp}

\usepackage{amscd}
\usepackage{tikz}
\usepackage{tikz-cd}
\usepackage{sidecap}
\usetikzlibrary{matrix,arrows}

\usepackage{mwe} \usepackage{subcaption}

\usepackage{float}
\floatstyle{boxed}
\restylefloat{table}
\restylefloat{figure}

\newtheoremstyle{break}{9pt}{9pt}{}{}{\bfseries}{.}{\newline}{}

\theoremstyle{break}
\newtheorem{theorem}{Theorem}[section]
\newtheorem{lemma}[theorem]{Lemma}

\newtheorem{remark}[theorem]{Remark}
\newtheorem{example}[theorem]{Example}

\newtheorem{corollary}[theorem]{Corollary}

\makeatletter
\newcommand\suchthat{\@ifstar
  {\mathrel{}\middle\vert\mathrel{}}
  {\mid}}
\makeatother

\usepackage{nicefrac}

\newcommand{\whitney}{\phi}

\newcommand{\vol}{\operatorname{vol}}
\newcommand{\eps}{{\epsilon}}

\newcommand{\kronecker}{{\delta}}

\newcommand{\inv}{{-1}}

\newcommand{\GL}{\operatorname{GL}}

\newcommand{\Sym}{\operatorname{Sym}}

\newcommand{\divergence}{\operatorname{div}}

\newcommand{\trace}{\operatorname{tr}}

\newcommand{\ext}{\operatorname{ext}}
\newcommand{\Id}{\operatorname{Id}}
\newcommand{\iunit}{\mathrm{i}}

\newcommand{\cartan}{{\mathsf d}}

\newcommand{\cartanlambda}{{\mathsf d}\lambda}
\newcommand{\nablalambda}{\nabla\lambda}

\newcommand{\Ned}{{\bfN\bf{d}}}
\newcommand{\RT}{{\bfR\bfT}}
\newcommand{\BDM}{{\bfB\bfD\bfM}}

\newcommand\xqedhere[2]{
  \rlap{\hbox to#1{\hfil\llap{\ensuremath{#2}}}}}

\newcommand{\Perm}{{\operatorname{Perm}}}

\DeclareMathOperator{\linhull}{span}

\newcommand{\bbC}{{\mathbb C}}

\newcommand{\bbF}{{\mathbb F}}

\newcommand{\bbN}{{\mathbb N}}

\newcommand{\bbR}{{\mathbb R}}

\newcommand{\bbZ}{{\mathbb Z}}

\newcommand{\bfB}{{\mathbf B}}

\newcommand{\bfD}{{\mathbf D}}

\newcommand{\bfM}{{\mathbf M}}
\newcommand{\bfN}{{\mathbf N}}

\newcommand{\bfR}{{\mathbf R}}

\newcommand{\bfT}{{\mathbf T}}

\newcommand{\calA}{{\mathcal A}}
\newcommand{\calB}{{\mathcal B}}

\newcommand{\calI}{{\mathcal I}}
\newcommand{\calJ}{{\mathcal J}}

\newcommand{\calP}{{\mathcal P}}
\newcommand{\calQ}{{\mathcal Q}}

\newcommand{\calS}{{\mathcal S}}

\newcommand{\frakr}{{\mathfrak r}}
\newcommand{\fraks}{{\mathfrak s}}

\begin{document}

\title
[Symmetry FEEC]
{Symmetry and Invariant Bases in\\ Finite Element Exterior Calculus}

\author{Martin W.\ Licht}

\address{\'Ecole Polytechnique F\'ed\'erale de Lausanne (EPFL), 1015 Lausanne, Switzerland}

\email{martin.licht@epfl.edu}

\thanks{This research was supported by the European Research Council through 
the FP7-IDEAS-ERC Starting Grant scheme, project 278011 STUCCOFIELDS.
This research was supported in part by NSF DMS/RTG Award 1345013 and DMS/CM Award 1262982.
The author would like to thank the Isaac Newton Institute for Mathematical Sciences, Cambridge, for support and hospitality during the programme ``Geometry, compatibility and structure preservation in computational differential equations'', where work on this paper was undertaken. This work was supported by EPSRC grant no EP/K032208/1.}

\subjclass[2000]{65N30,20C30}

\keywords{barycentric coordinates, finite element exterior calculus, monomial representation, representation theory of the symmetric group}

\begin{abstract}
We study symmetries of bases and spanning sets 
in finite element exterior calculus, using representation theory.
We want to know which vector-valued finite element spaces 
have bases invariant under permutation of vertex indices. 
The permutations of vertex indices correspond to the symmetry group of the simplex. 
That symmetry group is represented on simplicial finite element spaces by the pullback action.
We determine a natural notion of invariance 
and sufficient conditions on the dimension and polynomial degree 
for the existence of invariant bases. 
We conjecture that these conditions are necessary too. 
We utilize Djokovi\'c and Malzan's classification 
of monomial irreducible representations of the symmetric group,  
and show new symmetries of the geometric decomposition and canonical isomorphisms
of the finite element spaces. 
Explicit invariant bases with complex coefficients are constructed in dimensions two and three
for different spaces of finite element differential forms. 
 \end{abstract}

\maketitle

\section{Introduction} \label{sec:introduction}

The Lagrange finite element space over a simplex is easily defined for arbitrary polynomial degree.
The literature knows several canonical bases for higher-degree Lagrange spaces,
such as the standard nodal bases, the barycentric bases, and the Bernstein bases \cite{ainsworth2003hierarchic,ainsworth2011bernstein,kirby2012fast}.
A convenient feature of these canonical bases is their symmetry: the bases do not change if we re-number the vertices of the simplex.
Equivalently, they are invariant under pullback along the affine automorphisms of the simplex. 

While this convenient feature might easily be taken for granted, 
it fails to hold for vector-valued finite element spaces,
such as 
the Raviart-Thomas spaces,
Brezzi-Douglas-Marini spaces, and the N\'ed\'elec spaces of first and second kind
\cite{raviart1977mixed,brezzi1985two,nedelec1986new}. 
Indeed, even finding explicit bases for these vector-valued finite element spaces 
is a non-trivial topic and has only been addressed after the turn of the century
\cite{AFW1,AFWgeodecomp,gopalakrishnan2005nedelec,ervin2012computational,bentley2017explicit,licht2022basis}.
Whether an invariant basis of a given polynomial degree exists seems to be an intricate question: 
for example, 
while no such basis exists for the space of \emph{constant} vector fields over a triangle,
one easily finds such a basis for the \emph{linear} vector fields over a triangle. 
What pattern emerges for higher polynomial degrees?

The purpose of this article is to address this question: 
we present a natural notion of invariance and study the existence of invariant bases for vector-valued finite element spaces. 
As we demonstrate in this article, the existence of such bases seems to depend on the polynomial degree and the dimension. 
We identify sufficient conditions on these parameters for different families of finite element spaces, 
and we conjecture that these conditions are also necessary. 
To the author's best knowledge, no prior contributions to this fundamental topic exist.
We use a novel connection between finite element methods and group representation theory,
and a recursive construction of geometrically decomposed bases that is of independent interest. 
In what follows, 
we summarize the results and their prospective applications.
\\

We adopt the framework of finite element exterior calculus (FEEC, \cite{AFW1}),
which translates vector-valued finite element spaces into the calculus of differential forms. 
FEEC provides a unifying perspective on numerous aspects of finite element theory 
previously only known for special cases, such as 
convergence estimates \cite{AFW2,arnold2012mixed,arnold2013mixed,arnold2013finite}, 
approximation theory \cite{christiansen2008smoothed,HS1,licht2019smoothed,licht2016mixed},
and a posteriori error estimation \cite{demlow2014posteriori}.

The fundamental connection to representation theory is as follows. 
Every permutation of the vertices of a simplex 
corresponds to a unique affine automorphism of that simplex,
and the pullback along that automorphism acts on differential forms. 
Associating permutations to pullbacks in this manner preserves the group-structure. 
Hence representation theory enters the picture naturally: 
the permutation group is represented by the pullback operation over differential forms. 

It turns out that a satisfying theory of invariant bases involves complex coefficients. 
Thus our notion of invariance in this article means \emph{invariant under the action of the symmetric group
up to multiplication by complex units}. 
In the language of representation theory, we are interested under which conditions
the action of the symmetric group 
can be represented by \emph{monomial matrices} with real or complex coefficients \cite{ore1942theory}.
The transition to complex numbers reveals interesting structures: 
for example, the constant complex vector fields over a triangle have a basis
invariant up to multiplication by complex roots of unity. 
The language of differential forms is essential for our exposition.
\\

At the heart of our analysis is a new recursive construction of finite element bases,
which is interesting in its own right. 
We construct invariant bases for finite element spaces of higher polynomial degree 
via reduction to the case of lower polynomial degree. 
For that purpose, 
we analyze how simplicial symmetries interact 
with two fundamental concepts in finite element exterior calculus. 
On the one hand, 
we recall the geometric decomposition of the finite element spaces \cite{AFWgeodecomp},
\begin{align*}
 \calP_{r}\Lambda^{k}(T)
 =
 \bigoplus_{F \subseteq T}
 \ext^{k,r}_{F,T} \mathring\calP_{r}\Lambda^{k}(F)
 ,
 \quad
 \calP_{r}^{-}\Lambda^{k}(T)
 =
 \bigoplus_{F \subseteq T}
 \ext^{k,r,-}_{F,T} \mathring\calP^{-}_{r}\Lambda^{k}(F)
 . 
\end{align*}
We prove that the traces and extension operators commute with the pullbacks along simplicial symmetries. 
In particular, the geometrically decomposed bases are invariant if the individual summands are. 
Therefore, we can construct geometrically decomposed invariant bases from invariant bases with boundary conditions. 
On the other hand, we recall the canonical isomorphisms over an $n$-dimensional simplex $T$ \cite{AFW1}, namely,
\begin{gather*} 
 \calP_r\Lambda^{k}(T)
 \simeq
 \mathring\calP^{-}_{r+k+1}\Lambda^{n-k}(T)
 ,
 \quad
 \calP^{-}_{r+1}\Lambda^{k}(T)
 \simeq
 \mathring\calP_{r+k+1}\Lambda^{n-k}(T)
 .
\end{gather*} 
These isomorphisms are natural in the sense that they preserve the canonical spanning sets \cite{licht2022basis}. 
We prove that they commute with the simplicial symmetries up to signs and thus preserve invariant bases. 
In particular, we get invariant bases for finite element spaces with boundary conditions 
from invariant bases for finite element spaces of generally lower polynomial degree. 
Combining these two ideas, we recursively construct invariant bases,
using invariant bases that we have discovered in the necessary base cases. 
\\

The aforementioned base cases refer to the finite element spaces of lowest polynomial degree, that is, constant fields. 
Djokovi\'c and Malzan's classification of monomial irreducible representations of the symmetric group \cite{djokovic1975monomial}
shows that invariant bases for constant fields exist only in special cases: 
the scalar and volume forms, 
the constant differential forms up to dimension $3$,
and 
constant $2$-forms over $4$-simplices.
We outline the invariant bases for constant fields,
using vector calculus notation.
Let us write $\lambda_{i}$ for the barycentric coordinate associated to the $i$-th vertex of a tetrahedron or a triangle. 

Over a tetrahedron, the three vector fields 
\begin{gather} \label{intro:basis:constant1forms3D}
  \begin{gathered}
    \psi_{w} =   \nablalambda_0 - \nablalambda_1 + \nablalambda_2 - \nablalambda_3,
    \quad
    \psi_{p} =   \nablalambda_0 + \nablalambda_1 - \nablalambda_2 - \nablalambda_3,
    \\\ 
    \psi_{k} =   \nablalambda_0 - \nablalambda_1 - \nablalambda_2 + \nablalambda_3 
  \end{gathered}
\end{gather}
are a basis for the constant vector fields,
invariant up to signs under renumbering of vertices. 
Similarly, the three constant cross products 
\begin{align} \label{intro:basis:constant2forms3D}
  \psi_{w} \times \psi_{p},
  \quad
  \psi_{w} \times \psi_{k},
  \quad
  \psi_{p} \times \psi_{k}
\end{align}
are a basis for the constant pseudovector fields over a tetrahedron, 
again invariant up to signs under renumbering.
Over a triangle,
the transition to complex coefficients reveals 
that the two constant vector fields 
\begin{gather} \label{intro:basis:constant1forms2D}
  \begin{gathered}
    \theta_{0} = \nablalambda_0 + e^{2\iunit\pi/3}  \nablalambda_1 + e^{-2\iunit\pi/3} \nablalambda_2,
    \\ 
    \theta_{1} = \nablalambda_0 + e^{-2\iunit\pi/3} \nablalambda_1 +  e^{2\iunit\pi/3} \nablalambda_2 
  \end{gathered}
\end{gather}
are a basis for the \emph{complex} constant vector fields,
invariant under renumbering of vertices up to cubic roots of unity. 
We will also encounter a basis for the constant bivector fields over a four-dimensional hypertetrahedron
invariant up to quartic roots of unity.
\\

Starting from these base cases,
we recursively construct bases for finite element spaces of differential forms. 
Inspection of the base cases then reveals whether 
the construction yields bases invariant up to complex units or even up to signs:
this generally depends on the simplex dimension and the polynomial degree.

As a convenience for the reader, 
we summarize the application of our theory to common (real-valued) finite element spaces below.
Here, we use the language of vector analysis and the following notation:
$\lambda_i$ is again the barycentric coordinate associated with the $i$-th vertex,
$\phi_{ij}$ and $\phi_{ijk}$ are Whitney forms with the respective index sets, 
and $A(r,n)$ are the multiindices of weight $r$ in the indices $\{0, 1, \dots, n\}$;
see also Section~\ref{sec:preliminaries}.
The following finite element spaces have bases that are invariant up to sign changes 
under reordering of the vertices: 
\begin{itemize}
 \item 
 The Brezzi-Douglas-Marini space of degree $r$ over a triangle $T$,
 \begin{align*}
  \BDM_{r}(T) := \linhull\{ \; \lambda^{\alpha} \nablalambda_{i} \suchthat \alpha \in A(r,2), 0 \leq i \leq 2 \; \}
 \end{align*}
 if $r \notin 3 \bbN_{0}$.
 \item 
 The Raviart-Thomas space of degree $r$ over a triangle $T$,
 \begin{align*}
  \RT_{r}(T) := \linhull\{ \; \lambda^{\alpha} \phi_{ij} \suchthat \alpha \in A(r-1,2), 0 \leq i < j \leq 2 \; \}
 \end{align*}
 if $r \notin 3 \bbN_{0} + 2$.
\item 
 The divergence-conforming Brezzi-Douglas-Marini space of degree $r$ over a tetrahedron $T$,
 \begin{align*}
  \BDM_{r}(T) := \linhull\{ \; \lambda^{\alpha} \nablalambda_{i} \times \nablalambda_{j} \suchthat \alpha \in A(r,3), 0 \leq i < j \leq 3 \; \}
 \end{align*}
 if $r \in \{ 0, 1, 2, 4, 5, 8 \}$.
 \item 
 The divergence-conforming Raviart-Thomas space of degree $r$ over a tetrahedron $T$,
 \begin{align*}
  \RT_{r}(T) := \linhull\{ \; \lambda^{\alpha} \phi_{ijk} \suchthat \alpha \in A(r-1,3), 0 \leq i < j < k \leq 3 \; \}
 \end{align*}
 if $r \in \{ 0, 1, 2, 3, 4, 6, 7, 10 \}$.
 \item 
 The curl-conforming N\'ed\'elec space of the first kind of degree $r$ over a tetrahedron $T$, \begin{align*}
  \Ned_{r}^{\rm fst}(T) := \linhull\{ \; \lambda^{\alpha} \phi_{ij} \suchthat \alpha \in A(r-1,3), 0 \leq i < j \leq 3 \; \}
 \end{align*}
 if $r \in \{ 0, 1, 3, 4, 7 \}$. \item 
 The curl-conforming N\'ed\'elec space of the second kind of degree $r$ over a tetrahedron $T$, \begin{align*}
  \Ned_{r}^{\rm snd}(T) := \linhull\{ \; \lambda^{\alpha} \nablalambda_{i} \suchthat \alpha \in A(r,3), 0 \leq i \leq 3 \; \}
 \end{align*}
 if $r \in \{ 0, 1, 2, 4, 5, 8 \}$.
\end{itemize}
However, the complex-valued versions of these finite element spaces 
have bases invariant up to multiplication by cubic roots of unity,
irrespective of the polynomial degree. 

We conjecture that the above list of invariant bases is exhaustive and discuss some partial results in that regard.
As we show, if an invariant basis is already geometrically decomposed, 
in an intuitive sense formalized in the article, then we can reverse the recursive argument. 
For example, a basis of $\calP_{3}\Lambda^{1}(T)$ over a triangle cannot be geometrically decomposed and invariant up to signs
because otherwise we could construct a basis of $\calP_{0}\Lambda^{1}(T)$ invariant up to signs.
Similarly, a geometrically decomposed basis of $\calP_{r}\Lambda^{1}(T)$ over a tetrahedron that is invariant up to signs
gives rise to a basis of $\calP_{r}\Lambda^{1}(F)$ over a face $F$ invariant up to signs.
More generally, the recursive argument gives necessary conditions on bases to be both geometrically decomposed and invariant up to signs.
This does not rule out the existence of bases invariant up to signs that are not geometrically decomposed. 
\\

Bases for finite element spaces have been the subject of research for a long time. 
Bases for vector-valued finite element spaces, 
such as Brezzi-Douglas-Marini spaces, Raviart-Thomas spaces, or N\'ed\'elec spaces
have been stated explicitly only very recently \cite{AFW1,AFWgeodecomp,gopalakrishnan2005nedelec,ervin2012computational,bentley2017explicit,licht2022basis}. 
The choice of bases influences the condition numbers and sparsity properties of finite element matrices
\cite{ainsworth2003hierarchic,schoberl2005high,beuchler2012sparsity}. 
The present contribution continues previous research on bases and spanning sets in finite element exterior calculus \cite{licht2022basis}.

The invariance of bases under renumbering of the vertices
is a fundamental aspect of finite element spaces. 
It is not an issue for scalar-valued finite element spaces
but becomes highly nontrivial for vector-valued finite element spaces,
and no prior publication systematically discusses this topic. 
We remark that the seminal article of Arnold, Falk, and Winther \cite{AFW1}
utilized techniques of representation theory to classify 
the affinely invariant finite-dimensional vector spaces of polynomial differential forms.

Invariant bases are not only theoretically interesting  
but also of natural practical interest. 
Suppose that the triangulation has a significant share of regular triangles or tetrahedra:
if the bases exhibit the same geometric symmetries, 
then redundancies in the matrix coefficients 
can further reduce the assembly time of the local high-degree finite element matrices
and, perhaps more importantly, their memory footprint. 
Moreover, preliminary calculations indicate that invariant bases 
have natural orthogonality relations in such highly regular settings;
see also Example~\ref{example:interestingproperties}.
This hints at good conditioning of the bases 
in broader, more common settings when the triangulations are less regular. 
While we rely on a recursive basis construction, 
we notice that certain recursive structures already enable fast algorithms for finite element operators \cite{kirby2012fast,kirby2014low} (see also \cite{ainsworth2011bernstein}),
to reduce the computational complexity of higher-degree methods. 

Finite element bases with complex coefficients 
presumably suit best 
where complex coefficients emerge naturally, 
such as numerical electromagnetism 
or complex-shifted Laplacian Helmholtz solvers \cite{cools2017optimality}. 
Apart from computational studies, 
future research will study symmetric bases
on cubical finite element spaces \cite{arnold2013finite,arnold2011serendipity,gillette2019trimmed}
and the interaction of simplicial symmetries with resolutions of finite element spaces \cite{christiansen2016high}.
Moreover,
constructing symmetric degrees of freedom is a natural follow-up endeavor \cite{ainsworth2015bernstein}. 
\\

Some aspects of our analysis are of broader interest in representation theory. 
Our recursive construction showcases new aspects 
of the representation theory of the symmetric group. 
The notion of monomial representation is central to our contribution. 
However, monomial representations do not seem to be a standard topic 
in introductory textbooks on representation theory,
and only few articles approach constructive aspects of monomial representations
(see \cite{puschel1999constructive,puschel2002decomposing}). 
We also remark that groups of monomial matrices over finite fields 
have found use in cryptography and coding theory \cite{garcia2017monomial}.
The representation theory of groups has had various applications throughout numerical and computational mathematics,
such as in geometric integration theory \cite{StructPresDisc,stiefel2012group,munthe2016groups}
and artificial neural networks \cite{bloem2019probabilistic}. 
Our application of representation theory in finite element methods adds a new entry to that list.
\\

The remainder of this work is structured as follows. 
Important preliminaries on combinatorics, exterior calculus, and polynomial differential forms 
are summarized in Section \ref{sec:preliminaries}. 
We review elements of representation theory in Section \ref{sec:representationtheory}. 
In Section~\ref{sec:invariance} we establish first results on the coordinate transformation of polynomial differential forms. 
In Section~\ref{sec:lowestordercase}
we study invariant bases and spanning sets for lowest-degree finite element spaces. 
We discuss the symmetry properties of the canonical isomorphisms in Section~\ref{sec:canonicalisomorphism}. 
We discuss extension operators, geometric decompositions, and their symmetry properties 
in Section~\ref{sec:extension}. 
Putting these results together, 
Section~\ref{sec:recursion} discusses the recursive construction of invariant bases.

\section{Notation and Definitions}
\label{sec:preliminaries}

We introduce and review notions from combinatorics, simplicial geometry,
and differential forms over simplices. Parts of this section summarize results in \cite{licht2022basis}. 
We refer to Arnold, Falk, and Winther~\cite{AFW1,AFWgeodecomp}
and to Hiptmair~\cite{hiptmair2002finite} for further background 
on polynomial differential forms.

\subsection{Combinatorics} \label{subsec:combinatorics}

For $m, n \in \bbZ$ we write 
$[m:n] = \{ i \in \bbZ \suchthat m \leq i \leq n\}$,
which may be the empty set, 
and let $\eps(m,n) = 1$ if $m < n$ and $\eps(m,n) = -1$ if $m > n$. 
The set of all permutations of $[m:n]$ is written $\Perm(m:n)$
and we abbreviate $\Perm(n) := \Perm(0:n)$. 
We let $\eps(\pi) \in \{ -1, 1 \}$ be the \emph{sign} of any permutation $\pi \in \Perm(m:n)$.

We write $A(n)$ for the set of multiindices over $[0:n]$,
that is, the set of functions $\alpha : [0:n] \rightarrow \bbN_{0}$. 
For any $\alpha \in A(n)$,
\begin{gather*}
 \lvert \alpha \rvert := \sum_{i=0}^{n} \alpha(i),
 \quad
 [\alpha] := \left\{\; i \in [0:n] \suchthat \alpha(i) > 0 \;\right\},
\end{gather*}
and we let $\lfloor\alpha\rfloor$ denote the minimal element of $[\alpha]$  
provided that $[\alpha]$ is not empty, and $\lfloor\alpha\rfloor := \infty$ otherwise.
We let $A(r,n)$ be the set of those $\alpha \in A(n)$ for which $\lvert \alpha \rvert = r$.
The sum $\alpha+\beta$ of $\alpha, \beta \in A(n)$
is defined in the obvious manner.
We let $\delta_{p} : \bbZ \rightarrow \bbN$ be the function that equals $1$ at $p$
and is zero otherwise. 
When $\alpha \in A(r,n)$ and $p \in [0:n]$,
then $\alpha + p \in A(r+1,n)$ is notation for $\alpha + \delta_{p}$. 
Similarly, when $p \in [\alpha]$, then 
$\alpha - p \in A(r-1,n)$ is notation for $\alpha - \delta_{p}$.
\\

We let $\Sigma(a:b,m:n)$ be the set of strictly ascending mappings from $[a:b]$ to $[m:n]$.
We call those mappings \emph{alternator indices}. 
By convention, $\Sigma(a:b,m:n) := \{\emptyset\}$ whenever $a > b$.
For any $\sigma \in \Sigma(a:b,m:n)$ we let 
\begin{align*}
 [\sigma] := \left\{\; \sigma(i) \suchthat i \in [a:b] \; \right\},
\end{align*}
and we write $\lfloor\sigma\rfloor$ for the minimal element of $[\sigma]$  
provided that $[\sigma]$ is not empty, and $\lfloor\sigma\rfloor := \infty$ otherwise.
Furthermore, if $q \in [m:n] \setminus [\sigma]$, then we write $\sigma + q$
for the unique element of $\Sigma(a:(b+1),m:n)$ with image $[\sigma]\cup\{q\}$.
In that case, we also write $\eps(q,\sigma)$ for the sign of the permutation 
that sorts the sequence $q, \sigma(a), \dots, \sigma(b)$ in ascending order,
and we write $\eps(\sigma,q)$ for the sign of the permutation 
that sorts the sequence $\sigma(a), \dots, \sigma(b), q$ in ascending order.
Note also that $\eps(\sigma,q) = (-1)^{b-a} \eps(q,\sigma)$. 
Similarly, if $p \in [\sigma]$, then we write $\sigma - p$ 
for the unique element of $\Sigma(a:b-1,m:n)$ with image $[\sigma]\setminus\{p\}$.

We abbreviate $\Sigma(k,n) = \Sigma(1:k,0:n)$ and $\Sigma_{0}(k,n) = \Sigma(0:k,0:n)$. 
If $n$ is understood and $k, l \in [0:n]$,
then for any $\sigma \in \Sigma(k,n)$
we define $\sigma^c \in \Sigma_{0}(n-k,n)$
by the condition $[\sigma]\cup[\sigma^c] = [0:n]$,
and for any $\rho \in \Sigma_{0}(l,n)$
we define $\rho^c \in \Sigma(n-l,n)$
by the condition $[\rho]\cup[\rho^c] = [0:n]$.
In particular, $\sigma^{cc} = \sigma$ and $\rho^{cc} = \rho$.
We emphasize that $\sigma^c$ and $\rho^c$ depend on $n$,
which we suppress in the notation.
When $\sigma \in \Sigma(k,n)$ and $\rho \in \Sigma_{0}(l,n)$
with $[\sigma]\cap[\rho] = \emptyset$,
then $\eps(\sigma,\rho)$ is the sign of the permutation ordering the sequence
$\sigma(1),\dots,\sigma(k),\rho(0),\dots,\rho(l)$ in ascending order.

More generally, when $\tau : [a:b] \rightarrow [m:n]$ is an injective function, then 
we let $[\tau]$ be the range of $\tau$, 
and we let $\eps(\tau) \in \{-1,1\}$
be the sign of the permutation that sorts the sequence $\tau(a), \dots, \tau(b)$ in ascending order.

\subsection{Simplices} \label{subsec:simplices}

Let $n \in \bbN_{0}$. 
An \emph{$n$-dimensional simplex} $T$ is the convex closure 
of pairwise distinct affinely independent points $v^{T}_0, \dots, v^{T}_n$ in Euclidean space,
called the \emph{vertices} of $T$. 
We call $F \subseteq T$ a \emph{subsimplex} of $T$
if the set of vertices of $F$ is a subset of the set of vertices of $T$. 
We write $\imath(F,T) \colon F \rightarrow T$ for the set inclusion of $F$ into $T$.

As an additional structure, 
we assume that the vertices of all simplices are ordered.
For simplicity and without loss of generality, 
we assume that all simplices have vertices ordered compatibly to the order 
of vertices on their subsimplices. 
Suppose that $F$ is an $m$-dimensional subsimplex of $T$
with ordered vertices $v^{F}_0, \dots, v^{F}_m$. 
With a mild abuse of notation, we let $\imath(F,T) \in \Sigma_{0}(m,n)$
be the strictly ascending mapping defined by $v^{T}_{\imath(F,T)(i)} = v^{F}_{i}$.
Here, each vertex index of $F$ is mapped to the corresponding vertex index of $T$.

\subsection{Barycentric Coordinates and Differential Forms} \label{subsec:barycentriccoordinates}

Let $T$ be a simplex of dimension $n$. 
Following the notation of \cite{AFW1} and letting $k \in \bbZ$,
we denote by $\Lambda^{k}(T)$ the space of \emph{differential $k$-forms} over $T$ 
with real coefficients whose derivatives of all orders are smooth and bounded. 
Recall that these mappings take values in the $k$-th exterior power 
of the dual of the tangential space of the simplex $T$. 
In the case $k=0$, the space $\Lambda^{0}(T) = C^{\infty}(T)$
is just the space of smooth functions over $T$ with uniformly bounded derivatives. 
We use that $\Lambda^{k}(T) = \{0\}$ unless $0 \leq k \leq n$ without further mention. 

We write $\bbR\Lambda^{k}(T) = \Lambda^{k}(T)$
and let $\bbC\Lambda^{k}(T)$ denote the complexification of $\bbR\Lambda^{k}(T)$.
All the algebraic operations defined in the following 
apply to $\bbC\Lambda^{k}(T)$ completely analogously.

We recall the \emph{exterior product} $\omega \wedge \eta \in \Lambda^{k+l}(T)$
for $\omega \in \Lambda^{k}(T)$ and $\eta \in \Lambda^{l}(T)$
and that it satisfies $\omega \wedge \eta = (-1)^{kl} \eta \wedge \omega$. 
We let $\cartan \colon \Lambda^{k}(T) \rightarrow \Lambda^{k+1}(T)$
denote the \emph{exterior derivative}. 
It satisfies $\cartan\left( \omega \wedge \eta \right) = \cartan\omega \wedge \eta + (-1)^{k} \omega \wedge \cartan\eta$
for $\omega \in \Lambda^{k}(T)$ and $\eta \in \Lambda^{l}(T)$. 
We also recall that the integral $\int_{T} \omega$ of a differential $n$-form over $T$ is well-defined
upon the choice of any orientation of the simplex $T$.
\\

The \emph{barycentric coordinates} $\lambda^{T}_0, \dots, \lambda^{T}_n \in \Lambda^{0}(T)$
are the unique affine functions over $T$ that satisfy the \emph{Lagrange property}
\begin{align} \label{math:lagrangeproperty}
 \lambda^{T}_{i}( v_j ) = \delta_{ij},
 \quad
 i,j \in [0:n].
\end{align}
The barycentric coordinate functions of $T$ are linearly independent 
and constitute a partition of unity:
\begin{align} \label{math:partitionofunity}
 1 = \lambda^{T}_0 + \dots + \lambda^{T}_n
 .
\end{align}
We write $\cartanlambda_0^{T}, \cartanlambda_1^{T}, \dots, \cartanlambda_n^{T} \in \Lambda^{1}(T)$ 
for the exterior derivatives of the barycentric coordinates. 
These are differential $1$-forms and constitute a partition of zero: 
\begin{align} \label{math:partitionofzero}
 0 = \cartanlambda^{T}_0 + \dots + \cartanlambda^{T}_n
 .
\end{align}
It can be shown that this is, up to scaling, the only linear dependence 
between the exterior derivatives of the barycentric coordinate functions.
\\

Several classes of differential forms over $T$
that are expressed in terms of the barycentric coordinates and their exterior derivatives.
When $r \in \bbN_{0}$ and $\alpha \in A(r,n)$, 
then the corresponding \emph{barycentric polynomial} over $T$ is 
\begin{align} \label{math:definitionbarycentricpolynomial}
 \lambda_{T}^{\alpha}
 &:=
 \prod_{ i=0 }^{n} (\lambda_i^{T})^{\alpha(i)}
 .
\end{align}
When $a,b \in \bbN_{0}$ and $\sigma \in \Sigma(a:b,0:n)$,
the corresponding \emph{barycentric alternator} is 
\begin{align} \label{math:definitionbarycentricdifferentialform}
 \cartanlambda^{T}_{\sigma} 
 :=
 \cartanlambda^{T}_{\sigma(a)} \wedge\dots\wedge \cartanlambda^{T}_{\sigma(b)}
 .
\end{align}
Here, we treat the special case $\sigma = \emptyset$ by defining $\cartanlambda^{T}_{\emptyset} = 1$.

Whenever $a,b \in \bbN_{0}$ and $\rho \in \Sigma(a:b,0:n)$,
then the corresponding \emph{Whitney form} is 
\begin{align} \label{math:definitionwhitneyform}
 \whitney^{T}_{\rho}
 :=
 \sum_{p \in [\rho]} \eps(p,\rho-p) \lambda^{T}_{p} \cartanlambda^{T}_{\rho-p}
 .
\end{align}
In the special case 
that $\rho \colon [0:n] \rightarrow [0:n]$ is the single member of $\Sigma_{0}(n,n)$,
we write $\whitney_{T} := \whitney_{\rho}$ for the associated Whitney form. 
In what follows, the polynomials \eqref{math:definitionbarycentricpolynomial}
and their products with \eqref{math:definitionbarycentricdifferentialform} and \eqref{math:definitionwhitneyform}
are called \emph{barycentric differential forms} over $T$. 

We simplify the notation whenever there is no danger of ambiguity: \begin{align*}
  \lambda_i \equiv \lambda^{T}_i,
  \quad
  \lambda^{\alpha} \equiv \lambda_{T}^\alpha,
  \quad
  \cartanlambda_{\sigma} \equiv \cartanlambda_\sigma^{T},
  \quad
  \lambda_{\sigma} \equiv \lambda_\sigma^{T},
  \quad 
  \whitney_{\rho} \equiv \whitney_\rho^{T}
  .
\end{align*}
With our choice of notation, the simplex $T$ is always a superscript except for the barycentric monomials.

\subsection{Traces} \label{subsec:traces}

Let $T$ be an $n$-dimensional simplex and let $F \subseteq T$ be an $m$-dimensional subsimplex of $T$.
The trace from $T$ to $F$ is the mapping $\trace_{T,F} \colon \Lambda^{k}(T) \rightarrow \Lambda^{k}(F)$,
which is the pullback along the inclusion $\imath(F,T) \colon F \rightarrow T$ introduced above.
The trace commutes with the exterior derivative: 
$\trace_{T,F} \cartan \omega = \cartan \trace_{T,F} \omega$ for all $\omega \in \Lambda^{k}(T)$.

The trace does not depend on the order of the vertices.
However, taking into account the ordering of the vertices
provides explicit formulas for traces of barycentric differential forms.
Write $[\imath(F,T)]$ for the set of indices of those vertices of $T$ that are also vertices of $F$.

Consider $i \in [0:n]$.
If $i \notin [\imath(F,T)]$, then $v_{i}^{T}$ is a vertex of $T$ that is not in $F$,
and in that case $\trace_{T,F} \lambda^{T}_i = 0$.
If instead $i \in [\imath(F,T)]$,
then there exists $j \in [0:m]$ such that $i = \imath(F,T)(j)$,
and in that case $\trace_{T,F} \lambda^{T}_i = \lambda_j^{F}$. Analogous observations hold for the exterior derivatives of the barycentric coordinates.

Let $\alpha \in A(r,n)$ be a multiindex. 
If $[\alpha] \nsubseteq [\imath(F,T)]$, 
then $\trace_{T,F} \lambda_{T}^\alpha = 0$.
If instead $[\alpha] \subseteq [\imath(F,T)]$,
then there exists a unique $\widehat\alpha \in A(r,m)$ with
$\widehat\alpha = \alpha \circ \imath(F,T)$,
and hence 
\begin{align}
 \trace_{T,F} \lambda_{T}^{\alpha} = \lambda_{F}^{\widehat\alpha}
 .
\end{align}
Let $\sigma \in \Sigma(a:b,0:n)$ be an alternator index. 
If $[\sigma] \nsubseteq [\imath(F,T)]$, 
then $\trace_{T,F} \cartanlambda^{T}_\sigma = 0$.
If instead $[\sigma] \subseteq [\imath(F,T)]$,
then there exists a unique $\widehat\sigma \in \Sigma(a:b,0:m)$ with
$\imath(F,T) \circ \widehat\sigma = \sigma$,
and then 
\begin{align}
 \trace_{T,F} \cartanlambda_{\sigma}^{T} = \cartanlambda_{\widehat\sigma}^{F},
 \quad 
 \trace_{T,F} \whitney_{\sigma}^{T} = \whitney_{\widehat\sigma}^{F}. 
\end{align}

\subsection{Finite Element Spaces over Simplices} \label{subsec:spaces}

This subsection summarizes results about spanning sets and bases 
in finite element exterior calculus. Consider an $n$-dimensional simplex $T$,
a polynomial degree $r \in \bbZ$, 
and a form degree $k \in \bbZ$.
Suppose that either $\bbF = \bbR$ or $\bbF = \bbC$. 
We introduce the sets of polynomial differential forms 
\begin{subequations} \label{math:canonicalspanningset}
\begin{align}
 \label{math:canonicalspanningset:poly}
 \calS\calP_r\Lambda^{k}(T)
 &:=
 \left\{\;
  \lambda_{T}^{\alpha} \cartanlambda^{T}_{\sigma} 
  \suchthat* 
  \alpha \in A(r,n), 
  \;
  \sigma \in \Sigma(k,n)
 \;\right\},
 \\
 \label{math:canonicalspanningset:whitney}
 \calS\calP^{-}_r\Lambda^{k}(T)
 &:=
 \left\{\;
  \lambda_{T}^{\alpha} \whitney^{T}_{\rho}
  \suchthat* 
  \alpha \in A(r-1,n), \; \rho \in \Sigma_{0}(k,n)
 \;\right\},
 \\
\label{math:canonicalspanningset:polymathring}
 \calS\mathring\calP_r\Lambda^{k}(T)
 &:=
 \left\{\;
  \lambda_{T}^{\alpha}\cartanlambda^{T}_{\sigma}
  \suchthat* 
  \begin{array}{l}
   \alpha \in A(r,n), \; \sigma \in \Sigma(k,n),
   \\{}
   [\alpha]\cup[\sigma] = [0:n]
  \end{array}
 \;\right\},
 \\
 \label{math:canonicalspanningset:whitneymathring}
 \calS\mathring\calP^{-}_r\Lambda^{k}(T)
 &:=
 \left\{\; 
  \lambda_{T}^{\alpha}\whitney_{\rho}^{T}
  \suchthat* 
  \begin{array}{l}
   \alpha \in A(r-1,n), \; \rho \in \Sigma_{0}(k,n),
   \\{}
   [\alpha]\cup[\rho]=[0:n]
  \end{array}
 \;\right\}. 
\end{align}
\end{subequations}
These are important spanning sets: 
the linear hulls\footnote{These linear hulls are trivial if $r < 0$ or $k \notin [0:n]$.}
of the first two sets give rise to the standard finite element spaces 
of finite element exterior calculus with coefficients in the field $\bbF$: 
\begin{gather*}  \bbF\calP_r\Lambda^{k}(T)
 :=
 \linhull_{\bbF}\calS\calP_r\Lambda^{k}(T),
 \quad
 \bbF\calP^{-}_r\Lambda^{k}(T)
 :=
 \linhull_{\bbF}\calS\calP^{-}_r\Lambda^{k}(T).
\end{gather*}
Requiring the traces to vanish along the simplex boundary defines subspaces
\begin{subequations}
\begin{align*}
 \bbF\mathring\calP_r\Lambda^{k}(T)
 &:= 
 \left\{\; 
  \omega \in \bbF\calP_r\Lambda^{k}(T)
  \suchthat*  
  \forall F \subsetneq T : \trace_{T,F} \omega = 0 
 \;\right\}
 ,
 \\
 \bbF\mathring\calP_r^{-}\Lambda^{k}(T)
 &:= 
 \left\{\; 
  \omega \in \bbF\calP_r^{-}\Lambda^{k}(T)
  \suchthat*  
  \forall F \subsetneq T : \trace_{T,F} \omega = 0 
 \;\right\}
 .
\end{align*}
\end{subequations}
We know explicit spanning sets for these spaces as well. When $r \geq 1$, then 
\begin{gather*}  \bbF\mathring\calP_r\Lambda^{k}(T)
 =
 \linhull_{\bbF}\calS\mathring\calP_r\Lambda^{k}(T),
 \quad
 \bbF\mathring\calP^{-}_r\Lambda^{k}(T)
 =
 \linhull_{\bbF}\calS\mathring\calP^{-}_r\Lambda^{k}(T).
\end{gather*}
The first equation is also true when $r = 0$ and $k<n$, and the second equation is also true when $r = 0$;
the vector spaces are trivial in those cases.
The sets \eqref{math:canonicalspanningset} are called the \emph{canonical spanning sets}.
\footnote{This is a minor abuse of terminology: the space $\bbF\mathring\calP_0\Lambda^{n}(T)$ is one-dimensional but the corresponding set $\calS\mathring\calP_0\Lambda^{k}(T)$ is empty.}

The canonical spanning sets are generally not linearly independent and hence not bases. 
However, further constraining the indices in the canonical spanning sets
produces the following bases (see~\cite{licht2022basis}). 
When $r \geq 1$, we define the sets of barycentric differential forms 
\begin{subequations} \label{math:canonicalbases}
\begin{align}
 \label{math:canonicalbases:polybetterbasis}
 \calB\calP_r\Lambda^{k}(T)
 &:=
 \left\{\;
  \lambda_{T}^{\alpha} \cartanlambda^{T}_{\sigma}
  \suchthat* 
  \begin{array}{l}
   \alpha \in A(r,n), \; \sigma \in \Sigma(k,n), 
   \\
   \lfloor\alpha\rfloor \notin [\sigma]
  \end{array}
 \;\right\},
 \\
 \label{math:canonicalbases:whitneybasis}
 \calB\calP^{-}_r\Lambda^{k}(T)
 &:=
 \left\{\; 
  \lambda_{T}^{\alpha}\whitney_{\rho}^{T}
  \suchthat* 
   \begin{array}{l}
    \alpha \in A(r-1,n), \; \rho \in \Sigma_{0}(k,n),
    \\
    \lfloor\alpha\rfloor \geq \lfloor\rho\rfloor 
   \end{array} 
 \;\right\},
 \\
 \label{math:canonicalbases:polymathringbasis}
 \calB\mathring\calP_r\Lambda^{k}(T)
 &:=
 \left\{\;
  \lambda_{T}^{\alpha}\cartanlambda^{T}_{\sigma}
  \suchthat* 
  \begin{array}{l}
   \alpha \in A(r,n), \; \sigma \in \Sigma(k,n), 
   \\
   \lfloor\alpha\rfloor \notin [\sigma], \; [\alpha]\cup[\sigma] = [0:n]
  \end{array}
 \;\right\},
 \\
 \label{math:canonicalbases:whitneymathringbasis}
 \calB\mathring\calP^{-}_r\Lambda^{k}(T)
 &:=
 \left\{\;
  \lambda_{T}^{\alpha}\whitney_{\rho}^{T}
  \suchthat* 
   \begin{array}{l}
    \alpha \in A(r-1,n), \; \rho \in \Sigma_{0}(k,n),
    \\
    \lfloor\alpha\rfloor \geq \lfloor\rho\rfloor, \; [\alpha]\cup[\rho]=[0:n]
   \end{array} 
 \;\right\}.
\end{align}
\end{subequations}
A particular feature of these bases and spanning sets are their inclusion relations.
On the one hand, the bases are subsets of the spanning sets,
\begin{gather*}
 \calB\calP_r\Lambda^{k}(T)
 \subseteq  
 \calS\calP_r\Lambda^{k}(T), 
 \quad 
 \calB\calP_r^{-}\Lambda^{k}(T)
 \subseteq  
 \calS\calP_r^{-}\Lambda^{k}(T), 
 \\ 
 \calB\mathring\calP_r\Lambda^{k}(T)
 \subseteq  
 \calS\mathring\calP_r\Lambda^{k}(T),
 \quad 
 \calB\mathring\calP_r^{-}\Lambda^{k}(T)
 \subseteq  
 \calS\mathring\calP_r^{-}\Lambda^{k}(T).
\end{gather*}
On the other hand, the generators for the spaces with boundary conditions
are contained in the generators for the unconstrained spaces, 
\begin{gather*}
 \calS\mathring\calP_r\Lambda^{k}(T)
 \subseteq  
 \calS\calP_r\Lambda^{k}(T),
 \quad 
 \calS\mathring\calP_r^{-}\Lambda^{k}(T)
 \subseteq  
 \calS\calP_r^{-}\Lambda^{k}(T),
 \\ 
 \calB\mathring\calP_r\Lambda^{k}(T)
 \subseteq  
 \calB\calP_r\Lambda^{k}(T),
 \quad 
 \calB\mathring\calP_r^{-}\Lambda^{k}(T)
 \subseteq  
 \calB\calP_r^{-}\Lambda^{k}(T).
\end{gather*}
For any $\sigma \in \Sigma(k,n)$ and $\rho \in \Sigma_{0}(k,n)$ 
we let the bubble functions $\lambda^{T}_{\sigma} \in \calP_{k}(T)$ and $\lambda^{T}_{\rho} \in \calP_{k+1}(T)$ be defined by
\begin{align*}
  \lambda^{T}_{\sigma} := \lambda^{T}_{\sigma(1)} \lambda^{T}_{\sigma(2)} \cdots \lambda^{T}_{\sigma(k)},
  \quad 
  \lambda^{T}_{\rho}   := \lambda^{T}_{\rho(0)} \lambda^{T}_{\rho(1)} \cdots \lambda^{T}_{\rho(k)}.
\end{align*}
Note how this also defines the bubble functions $\lambda^{T}_{\sigma^{c}}$ and $\lambda^{T}_{\rho^{c}}$.
With those bubble functions,
we get yet another, explicit definition 
of the spanning sets \eqref{math:canonicalspanningset} and bases \eqref{math:canonicalbases} 
for the spaces with vanishing boundary traces:
\begin{subequations}
\begin{align}
 \calS\mathring\calP_r\Lambda^{k}(T)
 &=
 \left\{\;
  \lambda_{T}^{\beta} \lambda^{T}_{\sigma^{c}} \cartanlambda^{T}_{\sigma}
  \suchthat* 
   \beta \in A(r-n+k-1,n), \; \sigma \in \Sigma(k,n) 
 \;\right\},
 \\
 \calB\mathring\calP_r\Lambda^{k}(T)
 &=
 \left\{\;
  \lambda_{T}^{\beta} \lambda^{T}_{\sigma^{c}} \cartanlambda^{T}_{\sigma}
  \suchthat* 
  \begin{array}{l}
   \beta \in A(r-n+k-1,n), \; \sigma \in \Sigma(k,n)
   \\{}
   \lfloor\beta\rfloor \geq \lfloor\sigma^{c}\rfloor, \end{array}
 \;\right\},
 \\
 \calS\mathring\calP^{-}_r\Lambda^{k}(T)
 &=
 \left\{\;
  \lambda_{T}^{\beta} \lambda^{T}_{\rho^{c}} \whitney_{\rho}^{T}
  \suchthat* 
    \beta \in A(r-n+k-1,n), \; \rho \in \Sigma_{0}(k,n)
 \;\right\},
 \\
 \calB\mathring\calP^{-}_r\Lambda^{k}(T)
 &=
 \left\{\;
  \lambda_{T}^{\beta} \lambda^{T}_{\rho^{c}} \whitney_{\rho}^{T}
  \suchthat* 
  \begin{array}{l}
    \beta \in A(r-n+k-1,n), \; \rho \in \Sigma_{0}(k,n)
    \\{}
    \lfloor\rho\rfloor = 0, \end{array}
 \;\right\}. 
\end{align}
\end{subequations}
In the remainder of this document, we do not explicitly mention the field $\bbF$ when there is no danger of ambiguity.

\begin{remark}
The above bases and spanning sets for $\calP^{-}_r\Lambda^{k}(T)$ and $\mathring\calP^{-}_r\Lambda^{k}(T)$ are introduced in \cite{AFW1} and \cite{AFWgeodecomp}. The above bases and spanning sets for $\calP_r\Lambda^{k}(T)$ and $\mathring\calP_r\Lambda^{k}(T)$ are discussed in \cite{AFW1}, whereas \cite{AFWgeodecomp} introduces different bases. This subsection summarizes \cite[Section~4]{licht2022basis}, which contributes alternative proofs. 
\end{remark}

\section{Elements of Representation Theory} \label{sec:representationtheory}

In this section we gather elements of the representation theory of finite groups. 
We keep this rather concise and refer to the literature 
\cite{scott2012group,curtis1966representation,serre1977linear,jantzen2007representations,fulton2013representation}
for thorough expositions on representation theory. 
We are particularly interested in the notions of \emph{irreducible representations}, \emph{induced representations}, and \emph{monomial representations}. 
While the first two concepts are standard material in expositions on representation theory, 
the notion of monomial representation seems to have attracted much less attention yet. 

Throughout this section we fix a finite group $G$. 
The binary operation of the group is written multiplicatively. 
We let $e \in G$ denote the identity element of $G$
and we let $g^{\inv} \in G$ be the inverse element of any $g \in G$. 
Furthermore, 
we fix $\bbF \in \{ \bbR, \bbC \}$ in this section to be 
either the field of real numbers or the field of complex numbers. 
For any vector space $V$ over $\bbF$ we write $\GL(V)$ for its general linear group. 

A \emph{representation of $G$} is a group homomorphism $\frakr : G \rightarrow \GL(V)$
from $G$ into the general linear group of a vector space $V$. 
Necessarily, $\frakr(e) = \Id_V$ and for all $g, h \in G$ we have
$\frakr( g h ) = \frakr(g) \frakr(h)$ and $\frakr(g)^{\inv} = \frakr(g^{\inv})$. 
The \emph{dimension} of $\frakr$ is defined as the dimension of $V$,
and the representation $\frakr$ is called \emph{finite-dimensional} if $V$ is finite-dimensional. A representation is called \emph{faithful}
if it is a group monomorphism, that is,
only the unit of the group is mapped to the identity.

We call two representations $\frakr : G \rightarrow \GL(V)$ and $\fraks : G \rightarrow \GL(V)$ \emph{equivalent}
if there exists an isomorphism $J : V \rightarrow V$
such that $\fraks(g) = J^{-1} \frakr(g) J$ for all $g \in G$. 
In many circumstances, we are only interested in features of representations 
up to equivalence.

\begin{example}
 The most important example of a group in this article 
 is the group $\Perm(a:b)$ of permutations of the set $[a:b]$ for some $a, b \in \bbZ$. 
 The binary operation of the group is the composition. 
 We also recall the cycle notation: when $x_1, x_2, \dots, x_m \in [a:b]$ are pairwise distinct,
 then $\pi := (x_1 x_2 \dots x_m) \in \Perm(a:b)$ is the unique permutation that satisfies 
 \begin{align*}
  \pi(x_1) = x_2, \quad \pi(x_2) = x_3, \quad \dots \quad \pi(x_m) = x_1
 \end{align*}
 and leaves all other members of $[a:b]$ invariant. 
\end{example}

\begin{example} \label{example:basicexamplesofrepresentations}
 Let $G$ be any group and let $V$ be any vector space over the field $\bbF$.
 Then the mapping $\frakr : G \rightarrow \GL(V)$ that assumes the constant value $\Id_{V}$ 
 is a representation of $G$.
 This basic but important example is the so-called \emph{trivial representation of $G$}. 
For another basic example, 
 recall that every group $G$ generates the vector space $V = \bbF^{G}$ over $\bbF$.
 The mapping $\frakr : G \rightarrow \GL(V)$ such that $\frakr(g)h = gh$ for all $g, h \in G$
 is a representation of $G$. 
\end{example}

\subsection{Direct Sums, Subrepresentations, and Irreducible Representations}

We want to compose new representations from old representations. 
One way of doing so is the direct sum.
Let $\frakr : G \rightarrow \GL(V)$ and $\fraks : G \rightarrow \GL(W)$
be two representations of $G$. Their \emph{direct sum} 
\begin{align*}
 \frakr \oplus \fraks : G \rightarrow \GL( V \oplus W )
\end{align*}
is another representation of $G$ and is defined by 
\begin{align*}
 \left( (\frakr \oplus \fraks) g \right)(v,w) = ( \frakr(g) v, \fraks(g) w ),
 \quad 
 g \in G, \quad (v,w) \in V \oplus W.
\end{align*}
The definition of the direct sum extends to the case of more than two direct summands 
in the obvious manner. 
We are interested in how to conversely decompose a representation 
into direct summands. To study that topic, 
we introduce further terminology. 
\\

Let $\frakr : G \rightarrow \GL(V)$ be a representation. 
A subspace $W \subseteq V$ is called \emph{$\frakr$-invariant} if $\frakr(g) W = W$ for all $g \in G$.
Examples of $\frakr$-invariant subspaces are $V$ itself and the zero vector space. 
We call the representation $\frakr$ \emph{irreducible}
if the only $\frakr$-invariant subspaces of $V$ are the zero vector space and $V$ itself;
otherwise we call $\frakr$ \emph{reducible}. 

Suppose that $W \subseteq V$ is an $\frakr$-invariant subspace.
Then there exists a representation $\frakr^{W} : G \rightarrow \GL(W)$ in the obvious way.
We call $\frakr^{W}$ a \emph{subrepresentation} of $\frakr$.
The following result is known as Maschke's theorem \cite{maschke1899beweis}.

\begin{lemma}
 Let $\frakr : G \rightarrow \GL(V)$ be a finite-dimensional representation of $G$.
 Then there exist $\frakr$-invariant subspaces $V_{1}, \dots, V_{m} \subseteq V$
 such that 
 \begin{align*}
  V = V_{1} \oplus V_{2} \oplus \dots \oplus V_{m},
  \quad 
  \frakr = \frakr^{V_{1}} \oplus \frakr^{V_{2}} \oplus \dots \oplus \frakr^{V_{m}},
 \end{align*}
 and such that each $\frakr^{V_{i}}$ is irreducible. 
\end{lemma}

\begin{proof}
 If $\frakr$ is irreducible, then there is nothing to show.
 Otherwise, there exists an $\frakr$-invariant subspace $W \subset V$
 that is neither $V$ itself nor the trivial subspace. 
 We let $P : V \rightarrow W \subseteq V$ be any projection of $V$ onto $W$. Since $G$ is finite, we can define the linear mapping 
 \begin{align*}
  S : V \rightarrow V, \quad v \mapsto \lvert G \rvert^{-1} \sum_{ h \in G } \frakr(h)^{-1} P( \frakr(h) v ).
 \end{align*}
 One verifies that $S$ is again a projection onto $W$. 
 Furthermore, we see that $S( \frakr(g) v ) = \frakr(g) S( v )$ for all $g \in G$ and $v \in V$. 
 So $\ker(S)$ is $\frakr$-invariant.
 Since $V = W \oplus \ker(S)$ by linear algebra,
 we decompose $V$ into the direct sum of two non-trivial $\frakr$-invariant subspaces. 
 One then sees that $\frakr$ is the direct sum of the representations of $G$ over these subspaces. 
 Since $V$ is finite-dimensional, an induction argument over the dimension of $V$ shows the claim. 
\end{proof}

\subsection{Restrictions and Induced Representations}

Let $H \subset G$ be a subgroup of $G$. 
We recall that the cardinality of $H$ divides the cardinality of $G$,
and that the quotient $\lvert G \rvert/ \lvert H \rvert$ is called the \emph{index} of $H$ in $G$. 
Then we have a representation $\frakr_{H} : H \rightarrow V$
that is called the \emph{restriction} of $\frakr$ to the subgroup $H$. 
Generally, we cannot recover the original representation from its restriction to a subgroup. 
However, there exists a canonical way of inducing a representation of a group 
from any given representation of one of its subgroups.

Suppose that we have a representation $\fraks : H \rightarrow \GL(W)$
of the subgroup $H$ over the vector space $W$. 
First, we let $g_{1}, g_{2}, \dots, g_{M}$ be any list of representatives of the left cosets of $H$ in $G$, that is,
\begin{align*}
    \{ g_{1} H, g_{2} H, \dots, g_{M} H \} = \{ g H \suchthat g \in G \} 
    ,
\end{align*}
where necessarily $M = \lvert G \rvert/ \lvert H \rvert$ is in the \emph{index} of $H$ in $G$. 
We recall that 
for every $g \in G$ there exists a unique permutation $\tau_{g} \in \Perm(1:M)$
such that $g g_{i} \in g_{\tau(i)} H$. 
More specifically, there exists a unique $h_{g,i} \in H$
such that $g g_{i} = g_{\tau(i)} h_{g,i}$. 
We now define the vector space 
\begin{align*}
 V = \bigoplus_{i=1}^{M} W
\end{align*}
and define a representation $\frakr : G \rightarrow \GL(V)$ by setting componentwise 
\begin{align*}
 \frakr(g)(w_1,\dots,w_{M})_{\tau(i)}
 :=
 \fraks( h_{g,i} ) w_{i},
 \quad 
 1 \leq i \leq M, \quad w_{1},\dots,w_{M} \in W.
\end{align*}
In other words,
\begin{align*}
 \frakr(g)(w_1,\dots,w_{M})
 =
 \left( 
  \fraks( h_{g,\tau^{-1}(1)} ) w_{\tau^{-1}(1)},
  \dots,
  \fraks( h_{g,\tau^{-1}(M)} ) w_{\tau^{-1}(M)}
 \right).
\end{align*}
We call this the \emph{induced representation}. 
Conceptually, $V$ consists of $M$ copies of $W$, 
each associated to a coset representative $g_{i}$,
and the induced representation first applies the initial representation of $H$ componentwise
and then permutes the components.

We remark that the induced representation as defined above 
depends on the choice of representatives $g_{1}, g_{2}, \dots, g_{M}$ of the left cosets. Different sets of representatives lead to different induced representations, 
however, all those different representations are equivalent. 
Hence, technically, the literature defines induced representations only up to equivalence.
We refer to \cite[Chapter~12.5]{scott2012group} for further background and details.

\subsection{Monomial Representations and Invariant Sets}

A square matrix is called \emph{monomial}, or \emph{generalized permutation matrix}, if it is the product of a permutation matrix and an invertible diagonal matrix. 
Hence monomial matrices are the invertible matrices that have the non-zero pattern of a permutation matrix. 
A group representation $\frakr : G \rightarrow \GL(V)$ is called \emph{monomial}
if there exists a basis of $V$ with respect to which $\frakr(g)$ is a monomial matrix for each $g \in G$.

A representation of $G$ is called \emph{induced monomial}
if it is induced by a one-dimensional representation of a subgroup $H$ of $G$.
It is easy to see that every induced monomial representation is monomial. 
We remark that many authors use the term \emph{monomial} for what we call \emph{induced monomial}. 
For irreducible representations, 
being monomial and being induced monomial are equivalent \cite[Corollary 50.6]{curtis1966representation}.

\begin{lemma}
 If the representation $\frakr$ is irreducible and induced monomial,
 then $\frakr$ is monomial. 
\end{lemma}

We now introduce the notion of invariance that is central to the following studies.
To the author's best knowledge, 
the following is not standard terminology in the literature of representation theory.
Suppose that $\calQ \subseteq V$ is a set of $M$ pairwise different vectors of $V$, 
\begin{align*}
 \calQ = \left\{ \omega_{1}, \dots, \omega_{M} \right\}.
\end{align*}
We say that $\calQ$ is \emph{$\bbF$-invariant under $\frakr$} if 
for every $g \in G$
there exists a permutation $\tau \in \Perm(1:M)$
and a sequence of complex units $\chi_{1},\dots,\chi_{M} \in \bbF$
such that 
\begin{align*}
 \frakr(g) \omega_{i} = \chi_{i} \omega_{\tau(i)},
 \quad 
 1 \leq i \leq M.
\end{align*}
We notice that any $\bbR$-invariant subset of a real vector space 
gives rise to an $\bbR$-invariant subset of the complexification of that vector space.

\section{Notions of Invariance} \label{sec:invariance}

In this section we connect the preceding elements of representation theory
with finite element exterior calculus. 
We identify the pullback of barycentric differential forms along the affine automorphisms of a simplex 
as a representation of the symmetric group. 
Here and in all subsequent sections,
we let $\bbF \in \left\{ \bbR, \bbC \right\}$ be arbitrary
unless mentioned otherwise.
\\

We are particularly interested in the affine automorphisms of a simplex. 
Suppose that $T$ is an $n$-simplex with vertices $v_{0}, \dots, v_{n}$, respectively. 
For any permutation $\pi \in \Perm(n)$ there exists 
a unique affine diffeomorphism ${S}_{\pi} : T \rightarrow T$ such that 
\begin{align*}
 {S}_{\pi}( v_{i} ) = v_{ \pi^{-1}(i) }. 
\end{align*}
We let $\Sym(T)$ denote the \emph{symmetry group} of $T$,
which is the group of all affine automorphisms of $T$ and whose members we call \emph{simplicial symmetries}.
We say\footnote{When $\Sym(T)$ carries the composition as binary group operation, 
then the association $\pi \mapsto S_{\pi}$ is not a group homomorphism
but a so-called \emph{antihomomorphism}.} 
that the permutation $\pi$ \emph{induces} the simplicial symmetry $S_{\pi}$. 

Since the simplicial symmetries are also diffeomorphisms, 
we can pullback differential forms along them.
In the terminology of representation theory, we have representations 
\begin{align} \label{math:finiteelementrepresentation}
 \frakr : \Perm(n) \rightarrow \GL\left( \bbF\Lambda^{k}(T) \right), \quad \pi \mapsto S_{\pi}^{\ast},
\end{align}
that map permutations to the pullbacks along the corresponding simplicial symmetries.
We briefly verify that this is indeed a representation of the group $\Perm(n)$. 
For $\pi, \mu \in \Perm(n)$ we see
\begin{align*}
 S_{\pi \circ \mu}( v_{i} )
 =
 v_{(\pi \circ \mu)^{-1}(i)}
 =
 v_{\mu^{-1}\pi^{-1}(i)}
 =
 S_{\mu}( v_{\pi^{-1}(i)} )
 =
 S_{\mu}( S_{\pi}( v_{i} ) )
 . 
\end{align*}
Hence, $S_{\pi \circ \mu} = S_{\mu} S_{\pi}$. Consequently,
\begin{align*}
 S_{\pi \circ \mu}^{\ast} 
 = 
 \left( S_{\mu} S_{\pi} \right)^{\ast}
 =
 S_{\pi}^{\ast} S_{\mu}^{\ast}
 . 
\end{align*}
So the mapping \eqref{math:finiteelementrepresentation} does indeed define a group homomorphism
and thus is a representation of $\Perm(n)$. 
Of course, this representation is not finite-dimensional. 
\\

We are interested in the subrepresentation of the permutation group 
over spaces of polynomial differential forms.
We prepare this with several observations 
regarding the pullback operation on barycentric differential forms along ${S}_{\pi}$. 
For any $m, n \in \bbZ$, we write $\delta_{m,n}$ for the Kronecker delta.
For all $i,j \in [0:n]$ we observe that the pullback of the barycentric coordinates satisfies 
\begin{align*}
 \left( {S}_{\pi}^{\ast} \lambda_{i}^{T} \right) ( v_j )
 &= 
 \lambda^{T}_{i} \left( S_{\pi}( v_j ) \right)
 = 
 \lambda^{T}_{i} \left( v_{\pi^{-1}(j)} \right)
 = 
 \kronecker_{i,\pi^{-1}(j)}
 = 
 \kronecker_{\pi(i),j}
\end{align*}
Since the pullback along affine mappings preserves affine functions,
\begin{align} \label{math:transform:barycentriccoordinates}
 {S}_{\pi}^{\ast} \lambda^{T}_{i} = \lambda^{T}_{\pi(i)},
 \quad 
 {S}_{\pi}^{\ast} \cartanlambda^{T}_{i} = \cartan {S}_{\pi}^{\ast} \lambda^{T}_{i} = \cartanlambda^{T}_{\pi(i)}.
\end{align}
It follows that for any multiindex $\alpha \in A(n)$ we have 
\begin{align} \label{math:transform:monomials}
 {S}_{\pi}^{\ast} \lambda_{T}^{\alpha}
 = 
 {S}_{\pi}^{\ast} \prod_{i=0}^{n} \left( \lambda_{i}^{T} \right)^{\alpha(i)}
 =
 \prod_{i=0}^{n} \left( \lambda_{\pi(i)}^{T} \right)^{\alpha(i)}
 =
 \prod_{i=0}^{n} \left( \lambda_{i}^{T} \right)^{\alpha(\pi^{\inv}(i))}
 =
 \lambda_{T}^{\alpha \pi^{\inv}}
 .
\end{align}
For describing the pullback of barycentric differential forms along symmetry transformations, 
it suffices to consider basic alternators and Whitney forms. 
That is the content of the following two auxiliary lemmas. 

\begin{lemma} \label{lemma:transform:alternators}
    Let $k \in [1:n]$, $\sigma \in \Sigma(k,n)$ and $\pi \in \Perm(n)$. 
    Then 
    \begin{align} \label{math:transform:alternators}
        {S}_{\pi}^{\ast} \cartanlambda_{\sigma}^{T}
        =
        \eps(\pi\sigma)
        \cartanlambda_{\widehat\sigma}^{T}
        ,
    \end{align}
    where $\widehat\sigma \in \Sigma(k,n)$ such that $[\widehat\sigma] = [\pi\sigma]$.
\end{lemma}

\begin{proof}
    We observe that 
    \begin{align*}
        {S}_{\pi}^{\ast} \cartanlambda_{\sigma}^{T}
        &=
        {S}_{\pi}^{\ast}\cartanlambda_{\sigma(1)}^{T} \wedge \dots \wedge {S}_{\pi}^{\ast}\cartanlambda_{\sigma(k)}^{T}
        \\&=
        \cartanlambda_{\pi\sigma(1)}^{T} \wedge \dots \wedge \cartanlambda_{\pi\sigma(k)}^{T}
        =
        \eps(\pi\sigma)
        \cartanlambda_{\widehat\sigma(1)}^{T} \wedge \dots \wedge \cartanlambda_{\widehat\sigma(k)}^{T}
        =
        \eps(\pi\sigma)
        \cartanlambda_{\widehat\sigma}^{T}
        .
    \end{align*}
    Here, we have used that $\eps(\pi\sigma)$ is the sign of the permutation that brings the sequence 
    $\pi\sigma(1)$,$\pi\sigma(2)$,$\dots$,$\pi\sigma(k)$ into ascending order.
\end{proof}

\begin{lemma} \label{lemma:transform:whitneyforms}
    Let $k \in [0:n]$, $\rho \in \Sigma_{0}(k,n)$ and $\pi \in \Perm(n)$. 
    Then 
    \begin{align} \label{math:transform:whitneyforms}
        {S}_{\pi}^{\ast} \whitney_{\rho}^{T}
        =
        \eps(\pi\rho)
        \whitney_{\widehat\rho}^{T}
        ,
    \end{align}
    where $\widehat\rho \in \Sigma_{0}(k,n)$ such that $[\widehat\rho] = [\pi\rho]$.
\end{lemma}

\begin{proof}
    When $p \in [\rho]$, then $[\pi(\rho-p)] = [ \widehat\rho - \pi(p)]$.
    Using the definition of Whitney forms, the preceding lemma, 
    and a combinatorial identity to be proven shortly, 
    \begin{align*}
        {S}_{\pi}^{\ast} \whitney_{\rho}^{T}
        &=
        \sum_{p \in [\rho]} \eps(p,\rho-p) 
        \left( {S}_{\pi}^{\ast} \lambda^{T}_{p} \right) 
        \left( {S}_{\pi}^{\ast} \cartanlambda^{T}_{\rho-p} \right)
        \\&=
        \sum_{p \in [\rho]}
        \eps(p,\rho-p) 
        \eps(\pi(\rho-p))
        \lambda^{T}_{\pi(p)} 
        \cartanlambda^{T}_{ \widehat\rho - \pi(p) }
        \\&=
        \sum_{p \in [\rho]}
        \eps(\pi\rho)
        \eps(\pi(p),\widehat\rho-\pi(p)) 
        \lambda^{T}_{\pi(p)} 
        \cartanlambda^{T}_{ \widehat\rho - \pi(p) }
        =
        \eps(\pi\rho)
        \whitney_{\widehat\rho}^{T}
        .
    \end{align*}
We have used 
    $\eps(\pi\rho) \eps(\pi(p),\widehat\rho-\pi(p)) = \eps(p,\rho-p) \eps(\pi(\rho-p))$,
    which is shown as follows.
    Fix $p \in [\rho]$.
    Starting with the sequence $\rho(0),\rho(1),\dots,\rho(k)$,
	a permutation of sign $\eps(p,\rho-p)$ moves $p$ to the front of the sequence,  
    and after applying $\pi$ to each sequence entry,
    a permutation of sign $\eps(\pi(\rho-p))$ sorts the last $k$ entries in ascending order. 
That resulting sequence can also be constructed in a different way.
    Namely, we apply $\pi$ to each entry of the initial sequence 
    and let a permutation of sign $\eps(\pi\rho)$ sort the sequence $\pi\rho(0),\pi\rho(1),\dots,\pi\rho(k)$ in ascending order;
    then a permutation of sign $\eps(\pi(p),\widehat\rho-\pi(p))$ moves the entry $\pi(p)$ to the front position. 
\end{proof}

These observations suffice to completely describe 
the transformation of barycentric polynomial differential forms
along affine diffeomorphisms. 
Evidently, the finite element spaces studied in this article 
are invariant under the representation of the permutation group.
We have subrepresentations
\begin{gather*}
    \frakr : \Perm(n) \rightarrow \GL\left( \bbF\calP_r\Lambda^{k}(T) \right),
    \quad
    \frakr : \Perm(n) \rightarrow \GL\left( \bbF\calP_r^{-}\Lambda^{k}(T) \right),
    \\
    \frakr : \Perm(n) \rightarrow \GL\left( \bbF\mathring\calP_r\Lambda^{k}(T) \right),
    \quad
    \frakr : \Perm(n) \rightarrow \GL\left( \bbF\mathring\calP_r^{-}\Lambda^{k}(T) \right).
\end{gather*}
Now we apply the notion of $\bbF$-invariant set introduced in the preceding section. 
We say that a set $\calQ \subseteq \bbF\calP_{r}\Lambda^{k}(T)$ is $\bbF$-invariant 
if it is $\bbF$-invariant under the representation $\frakr$. 
To get a feel for this notion of invariance we provide a few examples. 
None of the following observations are a technical challenge. 

\begin{lemma} \label{lemma:invarianceof:spanningsets}
 Let $k,r \in \bbZ$ and $r \geq 0$. 
 The canonical spanning sets 
 $\calS\calP_{r}\Lambda^{k}(T)$, $\calS\calP_{r}^{-}\Lambda^{k}(T)$, $\calS\mathring\calP_{r}\Lambda^{k}(T)$,
 and $\calS\mathring\calP_{r}^{-}\Lambda^{k}(T)$
are $\bbR$-invariant. 
\end{lemma}

\begin{proof}
 This follows from the definitions of these sets
 together with 
 \eqref{math:transform:monomials}, 
 \eqref{math:transform:alternators}, 
 and 
 \eqref{math:transform:whitneyforms}.
\end{proof}

\begin{lemma}
 Let $k \in \bbZ$. 
 The basis $\calB\calP_1^{-}\Lambda^{k}(T)$ of the lowest-degree Whitney $k$-form is $\bbR$-invariant. 
\end{lemma}

\begin{proof}
 This follows from Lemma~\ref{lemma:invarianceof:spanningsets} 
 since $\calB\calP_1^{-}\Lambda^{k}(T) = \calS\calP_1^{-}\Lambda^{k}(T)$. 
\end{proof}

\begin{lemma} \label{lemma:invarianceof:simpleexamples} Let $r \in \bbN$. We have $\bbR$-invariant bases 
 \begin{gather*}
\calB\calP_{r}\Lambda^{0}(T),                 \quad
  \calB\calP_{r}^{-}\Lambda^{0}(T),             \quad 
  \calB\calP_{r}\Lambda^{n}(T),                 \quad
  \calB\calP_{r}^{-}\Lambda^{n}(T),             \\
  \calB\mathring\calP_{r}\Lambda^{0}(T),        \quad 
  \calB\mathring\calP_{r}^{-}\Lambda^{0}(T),    \quad 
  \calB\mathring\calP_{r}\Lambda^{n}(T),      \quad 
  \calB\mathring\calP_{r}^{-}\Lambda^{n}(T).    \end{gather*}
\end{lemma}

\begin{proof} 
Let $r \geq 1$. 
 In regard to $0$-forms, definitions imply the identities 
 \begin{align*}
  \calB\calP_{r}\Lambda^{0}(T)
  = \calS\calP_{r}\Lambda^{0}(T) 
  = \calS\calP^{-}_{r}\Lambda^{0}(T)
  = \calB\calP^{-}_{r}\Lambda^{0}(T),
  \\
  \calB\mathring\calP_{r}\Lambda^{0}(T) 
  = \calS\mathring\calP_{r}\Lambda^{0}(T)
  = \calS\mathring\calP^{-}_{r}\Lambda^{0}(T) 
  = \calB\mathring\calP^{-}_{r}\Lambda^{0}(T)
  .
 \end{align*}
 In regard to $n$-forms, one can show that 
 \begin{gather*}
  \calB\mathring\calP_{r}\Lambda^{n}(T)
  =
  \calB         \calP_{r}\Lambda^{n}(T)
  =
  \calS         \calP_{r}\Lambda^{n}(T), 
  \\ 
  \calB\mathring\calP^{-}_{r}\Lambda^{n}(T)
  =
  \calB         \calP^{-}_{r}\Lambda^{n}(T)
  =
  \calS         \calP^{-}_{r}\Lambda^{n}(T) 
  .
 \end{gather*}
To see the latter two equations, we note that $\Sigma_0(n,n)$ has only a single member $\rho$, which satisfies $[\rho]=[0:n]$. 
To see the former two equations,
 we recall that if $\alpha \in A(r,n)$ and $\sigma \in \Sigma(n,n)$
 with $\lfloor\alpha\rfloor \in [\sigma]$,
 then there exist unique $s \in \{1,-1\}$ and $q \in [0:n] \setminus [\sigma]$ 
 such that $\cartanlambda_{\sigma - \lfloor\alpha\rfloor + q} = s \cartanlambda_\sigma$. 
 Moreover, $[\sigma - \lfloor\alpha\rfloor + q] \cup [\alpha] = [0:n]$. 
 Thus,
 $\lambda_{T}^{\alpha} \cartanlambda^{T}_{\sigma} \in \calB\mathring\calP_{r}\Lambda^{n}(T)$. 
The desired $\bbR$-invariance of those sets
 follows from these identities together with Lemma~\ref{lemma:invarianceof:spanningsets}.
\end{proof}

While all the canonical spanning sets are $\bbR$-invariant,
we have identified only a few $\bbR$-invariant bases of finite element spaces.
The remainder of the exposition will address the following question:
under which circumstances do finite element spaces of differential forms 
have invariant bases in the sense of this subsection?

\section{Invariant Bases of Lowest Polynomial Degree} \label{sec:lowestordercase}

We commence our study of invariant bases with the case of the constant differential $k$-forms over an $n$-simplex. 
Already the lowest-degree case exhibits non-trivial features. 
It serves as the base case for recursively constructing invariant bases in the last section. 
We utilize some advanced results in the representation theory of the symmetric group.

\begin{lemma} \label{lemma:repisfaithfulirreducible}
 Let $T$ be an $n$-simplex and $k \in \bbN_{0}$. 
 The representation of $\Perm(n)$ on $\bbF\calP_{0}\Lambda^{k}(T)$ is irreducible. 
 It is faithful for $0 < k < n$. 
\end{lemma}

\begin{proof}
 If $0 < k < n$,
 it is easily seen that the representation is faithful 
 since only the identity element of $\Perm(n)$
 acts as the identity on $\bbF\calP_{0}\Lambda^{k}(T)$. 
 That the representation is irreducible 
 can be found in the literature \cite[Proposition 3.12]{fulton2013representation}. 
\end{proof}

We first consider the tetrahedron. 
We build an $\bbR$-invariant basis of $\bbR\calP_0\Lambda^{1}(T)$,
and then construct an $\bbR$-invariant basis for $\bbR\calP_0\Lambda^{2}(T)$ 
by taking the exterior power.

\begin{lemma} \label{lemma:invariantconstant:1forms3D}
 Let $T$ be a $3$-simplex. 
 An $\bbR$-invariant basis of $\bbR\calP_0\Lambda^{1}(T)$ is \begin{subequations} \label{math:invariantconstant:1forms3D}
 \begin{align} 
  \psi_{w} &=   \cartanlambda_0 - \cartanlambda_1 + \cartanlambda_2 - \cartanlambda_3,
  \\
  \psi_{p} &=   \cartanlambda_0 + \cartanlambda_1 - \cartanlambda_2 - \cartanlambda_3,
  \\
  \psi_{k} &=   \cartanlambda_0 - \cartanlambda_1 - \cartanlambda_2 + \cartanlambda_3.
\end{align}
 \end{subequations}
 In particular, this is also a $\bbC$-invariant basis of $\bbC\calP_{0}\Lambda^{1}(T)$. 
\end{lemma}

\begin{proof}
 An elementary calculation verifies that the set is a basis. 
 The permutation group $\Perm(3)$ is generated by the three cycles 
 $(01)$, $(02)$, and $(03)$. Direct computation verifies that 
 \begin{align*}
  S_{(01)}^{\ast} \psi_{w} = -\psi_{k}, 
  \quad 
  S_{(01)}^{\ast} \psi_{p} = +\psi_{p}, 
  \quad 
  S_{(01)}^{\ast} \psi_{k} = -\psi_{w}, 
  \\
  S_{(02)}^{\ast} \psi_{w} = +\psi_{w}, 
  \quad 
  S_{(02)}^{\ast} \psi_{p} = -\psi_{k},
  \quad 
  S_{(02)}^{\ast} \psi_{k} = -\psi_{p}, 
  \\
  S_{(03)}^{\ast} \psi_{w} = -\psi_{p}, 
  \quad 
  S_{(03)}^{\ast} \psi_{p} = -\psi_{w}, 
  \quad 
  S_{(03)}^{\ast} \psi_{k} = +\psi_{k}. 
\end{align*}
 Hence, this set is $\bbR$-invariant. 
\end{proof}

\begin{lemma} \label{lemma:invariantconstant:2forms3D}
 Let $T$ be a $3$-simplex. 
 An $\bbR$-invariant basis of $\bbR\calP_0\Lambda^{2}(T)$ is \begin{align} \label{math:invariantconstant:2forms3D}
  \psi_{w} \wedge \psi_{p},
  \quad
  \psi_{w} \wedge \psi_{k},
  \quad
  \psi_{p} \wedge \psi_{k}.
\end{align}
 In particular, this is also a $\bbC$-invariant basis of $\bbC\calP_{0}\Lambda^{2}(T)$. 
\end{lemma}

\begin{proof}
 We immediately see that these three $2$-forms are a basis of $\bbR\calP_0\Lambda^{2}(T)$. 
 Using the cycles $(01)$, $(02)$, and $(03)$ as in the previous proof,
 direct computation shows 
\begin{gather*}
  \begin{aligned}
    S_{(01)}^{\ast}\left( \psi_{w} \wedge \psi_{p} \right)
&=  \psi_{p} \wedge \psi_{k}
    ,
    &&&
    S_{(01)}^{\ast}\left( \psi_{w} \wedge \psi_{k} \right)
&= -\psi_{w} \wedge \psi_{k}
    ,
    \\
    S_{(02)}^{\ast}\left( \psi_{w} \wedge \psi_{p} \right)
&= -\psi_{w} \wedge \psi_{k}
    ,
    &&&
    S_{(02)}^{\ast}\left( \psi_{w} \wedge \psi_{k} \right)
&= -\psi_{w} \wedge \psi_{p}
    ,
    \\
    S_{(03)}^{\ast}\left( \psi_{w} \wedge \psi_{p} \right)
&= -\psi_{w} \wedge \psi_{p}
    ,
    &&&
    S_{(03)}^{\ast}\left( \psi_{w} \wedge \psi_{k} \right)
&= -\psi_{p} \wedge \psi_{k}
    ,
  \end{aligned}
    \\
  \begin{aligned}
S_{(01)}^{\ast}\left( \psi_{p} \wedge \psi_{k} \right)
&=  \psi_{w} \wedge \psi_{p}
    ,
    \\ S_{(02)}^{\ast}\left( \psi_{p} \wedge \psi_{k} \right)
&= -\psi_{p} \wedge \psi_{k}
    ,
    \\ S_{(03)}^{\ast}\left( \psi_{p} \wedge \psi_{k} \right)
&= -\psi_{w} \wedge \psi_{k}
    ,
\end{aligned}
 \end{gather*}
 Hence, this set is $\bbR$-invariant.
\end{proof}

Next we inspect the triangle, where the situation is more complicated: 
we need to consider not only real but also complex coefficients. 

\begin{lemma} \label{lemma:invariantconstant:1forms2D}
 Let $T$ be a $2$-simplex. 
 A $\bbC$-invariant basis of $\bbC\calP_0\Lambda^{1}(T)$ is \begin{gather} \label{math:invariantconstant:1forms2D}
  \theta_{0} = \cartanlambda_0 + \xi_3     \cartanlambda_1 + \xi_3^{2} \cartanlambda_2,
  \quad 
  \theta_{1} = \cartanlambda_0 + \xi_3^{2} \cartanlambda_1 + \xi_3     \cartanlambda_2,
\end{gather}
 where $\xi_3 = \exp(2\iunit\pi/3 )$ is the cubic root of unity.
 $\bbR\calP_0\Lambda^{1}(T)$ has no $\bbR$-invariant basis. 
\end{lemma}

\begin{proof}
 We easily check that the two vectors constitute a basis and that 
 \begin{align*}
  S_{(01)}^{\ast} \theta_{0} = \xi_3^{ } \theta_{1},
  \quad  
  S_{(01)}^{\ast} \theta_{1} = \xi_3^{2} \theta_{0},
  \\ 
  S_{(02)}^{\ast} \theta_{0} = \xi_3^{2} \theta_{1},
  \quad 
  S_{(02)}^{\ast} \theta_{1} = \xi_3^{ } \theta_{0},
 \end{align*}
 where we have used the cycles $(01), (02) \in \Perm(2)$.
 Since those are generators of $\Perm(2)$,
it follows that $\{ \theta_{0}, \theta_{1} \}$ is a $\bbC$-invariant basis
 of $\bbC\calP_0\Lambda^{1}(T)$.
 
 Suppose that $\bbC\calP_0\Lambda^{1}(T)$ has an $\bbR$-invariant basis. 
 Since our representation of $\Perm(2)$ over $\bbC\calP_0\Lambda^{1}(T)$ is faithful
 by Lemma~\ref{lemma:repisfaithfulirreducible}, 
 it then follows that $\Perm(2)$ is isomorphic to a subgroup
 of the group of $2 \times 2$ signed permutation matrices. 
 The latter group has order $8$ whereas $\Perm(2)$ has order $6$.
This contradicts the well-known fact that the order of a group is divided by the orders of their subgroups.  
So $\bbC\calP_0\Lambda^{1}(T)$ has no $\bbR$-invariant basis.
\end{proof}

Seemingly serendipitously, 
we present a $\bbC$-invariant basis for the constant bivector fields over a $4$-simplex.

\begin{lemma} \label{lemma:invariantconstant:2forms4D}
 Let $T$ be a $4$-simplex.
 Define $\tau, \kappa \in \Perm(4)$ by $\tau = (01)$ and $\kappa = (01234)$. 
 We abbreviate $\cartanlambda_{ij} = \cartanlambda_{i} \wedge \cartanlambda_{j}$ for $0 \leq i, j \leq 4$.
 Then a $\bbC$-invariant basis of $\bbC\calP_0\Lambda^{2}(T)$ is given by 
\begin{align*}
  \zeta_{0} &= 
    \left( 
	\cartanlambda_{01}
	+
    \cartanlambda_{12}
	+
    \cartanlambda_{23}
	+
    \cartanlambda_{34}
	+
    \cartanlambda_{40}
	\right) 
    \\&\qquad+
	\iunit
	\left( 
	\cartanlambda_{02}
	+
    \cartanlambda_{24}
	+
    \cartanlambda_{41}
	+
    \cartanlambda_{13}
	+
    \cartanlambda_{30}
	\right) 
 \end{align*}
 and 
 \begin{gather*}
  \zeta_{1} = S_{\tau}^{\ast} \zeta_{0},
  \quad 
  \zeta_{2} = S_{\kappa}^{\ast} \zeta_{1},
  \quad 
  \zeta_{3} = S_{\kappa}^{\ast} \zeta_{2},
  \quad 
  \zeta_{4} = S_{\kappa}^{\ast} \zeta_{3},
  \quad 
  \zeta_{5} = S_{\kappa}^{\ast} \zeta_{4}.
\end{gather*}
\end{lemma}

\begin{proof}
 Recall that $\tau$ and $\kappa$ are generators of the group $\Perm(4)$.
 One easily checks that
 \begin{gather*}
  \zeta_{0} = S_{\kappa}^{\ast} \zeta_{0},
  \quad 
  \zeta_{1} = S_{\kappa}^{\ast} \zeta_{5},
  \quad 
  -\iunit \zeta_{3} = S_{\tau}^{\ast} \zeta_{2},
  \\ 
   \iunit \zeta_{2} = S_{\tau}^{\ast} \zeta_{3},
  \quad 
   \iunit \zeta_{5} = S_{\tau}^{\ast} \zeta_{4},
  \quad 
  -\iunit \zeta_{4} = S_{\tau}^{\ast} \zeta_{5}.
 \end{gather*}
 It follows that these vectors are a $\bbC$-invariant set. 
 That they are a basis is verified by elementary calculations. 
 For example, we expand these forms in terms of a basis of $\bbC\calP_0\Lambda^{2}(T)$, 
 and that the $6 \times 6$ matrix of the coefficients has non-zero determinant.
\end{proof}

\begin{remark}
 Whereas Djokovi\'c and Malzan's results \cite{djokovic1975monomial} include that a monomial representation 
 of $\Perm(4)$ over $\bbC\calP_{0}\Lambda^{2}(T)$ exists,
 they do not state an explicit basis and their argument is not immediately constructive. 
For that reason, we review how the aforementioned basis can be found.
 
 Recall that $\Perm(4)$ is generated by the two cycles $(01)$ and $(01234)$. 
 The $5$-cycle $(01234)$ is represented by a generalized permutation matrix of size $6 \times 6$,
 and so that matrix has the non-zero structure of a $6 \times 6$ permutation matrix of order $5$. 
 In particular, one of the $\bbC$-invariant basis vectors must be invariant under the cyclic vertex permutation.
We make the initial ansatz that the monomial matrices have coefficients in the quartic roots of unity.
 Via machine assisted brute-force search one finds $4$ different vectors with that invariance property, 
 up to multiplication by complex units.

 With the additional ansatz that the $2$-cycle $(01)$ maps these invariant forms 
 into the orbit of the aforementioned $5$-cycle, 
 one constructs five more vectors of the supposed basis.
 One then checks manually their linear independence and their $\bbC$-invariance. 
Up to multiplication by complex units, 
 this procedure only leaves the basis in Lemma~\ref{lemma:invariantconstant:2forms4D} and its complex conjugate.
\end{remark}

\begin{remark} 
 The following observations have been suggested by the anonymous referee 
 and are included as a service for the reader. They shed new light onto the basis vectors above. 
 For any $5$-cycle $g=(abcde) \in \Perm(0:4)$, we let 
 \begin{align*}
  \zeta_{g} = \omega_{g} + \iunit \omega_{g^2},
  \quad 
  \omega_{g} = \cartanlambda_{ab} + \cartanlambda_{bc} + \cartanlambda_{cd} + \cartanlambda_{de} + \cartanlambda_{ea}
  .
 \end{align*}
 We immediately observe $\omega_{g^{-1}} = - \omega_{g}$. Together $g^5 = e$, one calculates 
 \begin{gather*}
  \zeta_{g^{2}}
= 
  - \iunit \zeta_{g},
  \quad 
  \zeta_{g^{3}}
=
  \iunit \zeta_{g},
  \quad 
  \zeta_{g^{4}}
=
  - \zeta_{g}.
 \end{gather*}
 So the $5$-cycles generated by $g$ induce the same $\zeta_{g}$ up to quartic roots of unity. 
 It is clear that relabeling the simplex vertices will send $\zeta_{g}$ to $\zeta_{g'}$ for some $5$-cycle $g' \in \Perm(0:4)$. 
 
 Each $5$-cycle $g \in \Perm(0:4)$ generates a cyclic subgroup of order $5$.
 Since the entire group contains $4!$ different $5$-cycles,
 we see that every $5$-cycle must be belong to exactly one of six different cyclic subgroups.
 Upon choosing six $5$-cycles $g_1,\dots,g_6$
 that generate the six different subgroups,
 the corresponding forms $\zeta_{g_1},\dots,\zeta_{g_6}$ are invariant up to complex units. 
 The $5$-cycles 
 \begin{align*}
  (01234), \quad (10234), \quad (02134), \quad (01324), \quad (01243), \quad (41230)
 \end{align*}
 are such a choice of generators. They induce the basis stated in Lemma~\ref{lemma:invariantconstant:2forms4D}. 
\end{remark}

We have already pointed out Djokovi\'c and Malzan's contribution \cite{djokovic1975monomial}
on monomial representations of the symmetric group.  
The invariant bases constructed in this section concretize their results. 
Apart from the constant scalar and volume forms, for which $\bbR$-invariant bases are obvious, 
the bases found above already are exhaustive examples:
no other spaces of constant differential forms over simplices of any dimension 
allows for $\bbC$-invariant bases. 
That is the content of the following result.

\begin{theorem} \label{lemma:invariantconstant:classification}
 The space $\calP_{0}\Lambda^{k}(T)$ has a $\bbC$-invariant basis 
 only if $k=0$ or if $k=n$ or if $\dim T \leq 3$ or if $k=2$ with $\dim(T) = 4$.
\end{theorem}

\begin{proof}
    We recall that the representations of $\Perm(n)$ over $\calP_{0}\Lambda^{k}(T)$ are irreducible. 
    Djokovi\'c and Malzan have shown \cite[Theorem~1]{djokovic1975monomial} that 
    the only induced monomial irreducible representation of the group $\Perm(n)$
    over spaces of constant differential forms
    are the trivial and the alternating representations,
    which corresponds to the group action on the space of constant functions and constant volume forms,
    the irreducible representations of $\Perm(2)$ and $\Perm(3)$,
    and an irreducible representation of $\Perm(4)$ on the space $\calP_{0}\Lambda^{2}(T)$
    for any $4$-dimensional simplex $T$. 
    Thus Lemma~\ref{lemma:invarianceof:simpleexamples},
    and Lemmas~\ref{lemma:invariantconstant:1forms3D}, \ref{lemma:invariantconstant:2forms3D},
     \ref{lemma:invariantconstant:1forms2D},~and~\ref{lemma:invariantconstant:2forms4D}
    cover the irreducible representations of symmetric groups over constant differential forms.\footnote{The group $\Perm(0:3)$ also has a two-dimensional induced monomial irreducible (and hence monomial) representation, but this is of no interest in our applications.}
All other irreducible representations of $\Perm(n)$ are not induced monomial. 
    Since induced monomial irreducible representations are monomial, 
    the theorem follows. 
\end{proof}

\section{Canonical Isomorphisms} \label{sec:canonicalisomorphism}

In this section we review the interaction of simplicial symmetries 
with the canonical isomorphisms in finite element exterior calculus.
We show that the isomorphisms preserve $\bbF$-invariance of sets.
These isomorphisms were discussed in \cite{AFW1} and also \cite{christiansen2016high};
we follow the discussion in \cite{licht2022basis},
where it is shown that these isomorphisms can be described in terms of the canonical spanning sets. 
In that sense, the isomorphisms are natural for finite element exterior calculus.
\\

Let $k,r \in \bbN_{0}$ with $r \geq 0$. Recall the canonical isomorphisms
\begin{subequations}
\begin{gather}
 \calI_{k,r} : \calP_{r}\Lambda^{k}(T) \rightarrow \mathring\calP^{-}_{r+k+1}\Lambda^{n-k}(T)
 ,
 \\
 \calJ_{k,r} : \calP_{r+1}^{-}\Lambda^{k}(T) \rightarrow \mathring\calP_{r+k+1}\Lambda^{n-k}(T)
 .
\end{gather}
\end{subequations}
These are uniquely defined by the identities 
\begin{subequations}
\begin{align}
 \calI_{k,r}\left( 
   \lambda^{\alpha} \cartanlambda_{\sigma}
 \right)
 &=
 \eps(\sigma,\sigma^{c})
 \lambda^{\alpha} \lambda_{\sigma} \whitney_{\sigma^{c}}
 ,
 \quad 
 \alpha \in A(r,n), \quad \sigma \in \Sigma(k,n)
 ,
 \\
 \calJ_{k,r}\left( 
  \lambda^{\alpha} \whitney_{\rho}
 \right)
 &=
 \eps(\rho^{c},\rho)
 \lambda^{\alpha} \lambda_{\rho} \cartanlambda_{\rho^{c}}
 ,
 \quad 
 \alpha \in A(r,n), \quad \rho \in \Sigma_{0}(k,n)
 .
\end{align}
\end{subequations}
Note that these two identities prescribe the values of $\calI_{k,r}$ and $\calJ_{k,r}$ 
over the canonical spanning sets, which are not necessarily linearly independent. 
However, one can show that these definitions nevertheless yield well-defined $\bbF$-linear mappings \cite{licht2022basis}. 

\begin{remark}
 The seminal idea of these isomorphisms is 
 mapping between finite element spaces without and with boundary conditions
 via multiplication by monomial ``bubble'' functions.
 For example, in the case $k=n$, 
 we have $\calJ_{n,r}( f \vol_T ) = \lambda_{0}\lambda_{1}\cdots\lambda_{n}\cdot f$
 for all $f \in \calP_{r}\Lambda^{0}(T)$,
 where $\vol_T$ denotes the volume form of $T$.
 The canonical isomorphisms generalize that idea. 
\end{remark}

The following lemma shows that the canonical isomorphisms commute with the simplicial symmetries 
up to sign changes. 

\begin{theorem} \label{theorem:isomorphisminvariance}
 Let $\pi \in \Perm(n)$ and $S_{\pi} \in \Sym(T)$. 
 Then
 \begin{gather*}
  S_{\pi}^{\ast} \calI_{k,r} = \eps(\pi) \calI_{k,r} S_{\pi}^{\ast},
  \quad 
  S_{\pi}^{\ast} \calJ_{k,r} = \eps(\pi) \calJ_{k,r} S_{\pi}^{\ast},
  \\
  S_{\pi}^{\ast} \calI_{k,r}^{-1} = \eps(\pi) \calI_{k,r}^{-1} S_{\pi}^{\ast},
  \quad 
  S_{\pi}^{\ast} \calJ_{k,r}^{-1} = \eps(\pi) \calJ_{k,r}^{-1} S_{\pi}^{\ast}.
 \end{gather*}
\end{theorem}

\begin{proof}
 Let $\alpha \in A(r,n)$, $\sigma \in \Sigma(k,n)$, and $\pi \in \Perm(n)$.
 We let $\widehat\sigma \in \Sigma(k,n)$ satisfy $[\widehat\sigma] = [\pi\sigma]$. 
 We also write $\widehat \alpha = \alpha \pi^{-1}$. 
Using the results of Section~\ref{sec:invariance},
 direct calculation now shows that  
 \begin{align*}
  S_{\pi}^{\ast} \calI_{k,r} \left( \lambda^{\alpha} \cartanlambda_{\sigma} \right)
  &=
  \eps(\sigma,\sigma^c) S_{\pi}^{\ast} \left( \lambda^{\alpha} \lambda_{\sigma} \whitney_{\sigma^c} \right)
\\&=
  \eps(\sigma,\sigma^c) \eps(\pi\sigma^c)
  \lambda^{\widehat \alpha} \lambda_{\widehat\sigma} \whitney_{\widehat\sigma^c}
  \\&=
  \eps(\sigma,\sigma^c) \eps(\pi\sigma^c) \eps(\widehat\sigma,\widehat\sigma^c)
  \calI_{k,r} ( \lambda^{\widehat \alpha} \cartanlambda_{\widehat\sigma} )
\\&=
  \eps(\sigma,\sigma^c) \eps(\pi\sigma^c) \eps(\widehat\sigma,\widehat\sigma^c) \eps(\pi\sigma)
  \calI_{k,r} S_{\pi}^{\ast} ( \lambda^{\alpha} \cartanlambda_{\sigma} )
  . 
 \end{align*}
 We now use the following combinatorial observation. 
 Starting with the sequence $0,1,\dots,n$,
 a first permutation of sign $\eps(\widehat\sigma,\widehat\sigma^c)$
 produces the sequence $\widehat\sigma$ followed by $\widehat\sigma^c$.
 Two further permutations of signs $\eps(\pi\sigma)$ and $\eps(\pi\sigma^c)$, respectively, 
 bring these two subsequences into the form $\pi\sigma$ followed by $\pi\sigma^c$. 
 A final permutation of sign $\eps(\sigma,\sigma^c)$ 
 produces the sequence $\pi(0),\pi(1),\dots,\pi(n)$.
 Hence 
 \begin{align*}
  \eps(\pi\sigma,\pi\sigma^c)
  \eps(\pi\sigma) \eps(\pi\sigma^c) 
  \eps(\sigma,\sigma^c)
  =
  \eps(\pi)
  . 
 \end{align*}
 The desired identity for the first canonical isomorphism follows. 

 Analogous calculations work for the other isomorphism.
 Let $\rho \in \Sigma_{0}(k,n)$ and $\widehat\rho \in \Sigma_{0}(k,n)$ satisfy $[\widehat\rho] = [\pi\rho]$. 
 Let $\alpha \in A(r,n)$ and $\widehat \alpha = \alpha \pi^{-1}$. 
 Then
 \begin{align*}
  S_{\pi}^{\ast} \calJ_{k,r} \left( \lambda^{\alpha} \whitney_{\rho} \right)
  &=
  \eps(\rho^{c},\rho)
  S_{\pi}^{\ast} \left( \lambda^{\alpha} \lambda_{\rho} \cartanlambda_{\rho^{c}} \right)
\\&=
  \eps(\rho^{c},\rho) \eps(\pi\rho^c)
  \lambda^{\widehat\alpha} \lambda_{\widehat\rho} \cartanlambda_{\widehat\rho^{c}} 
  \\&=
  \eps(\rho^{c},\rho) \eps(\pi\rho^c) \eps(\widehat\rho^{c},\widehat\rho)
  \calJ_{k,r} ( 
    \lambda^{\widehat\alpha} \whitney_{\widehat\rho} 
  )
\\&=
  \eps(\rho^{c},\rho) \eps(\pi\rho^c) \eps(\widehat\rho^{c},\widehat\rho) \eps(\pi\rho)
  \calJ_{k,r} S_{\pi}^{\ast} ( \lambda^{\alpha} \whitney_{\rho} )
  \\&=
  \eps(\pi)
  \calJ_{k,r} S_{\pi}^{\ast} ( \lambda^{\alpha} \whitney_{\rho} )
  . 
 \end{align*}
 Finally, 
we observe 
 \begin{gather*}
  \calI_{k,r}^{-1} S_{\pi}^{\ast} 
  = 
  \calI_{k,r}^{-1} S_{\pi}^{\ast} \calI_{k,r} \calI_{k,r}^{-1} 
  = 
  \eps(\pi)
  \calI_{k,r}^{-1} \calI_{k,r} S_{\pi}^{\ast} \calI_{k,r}^{-1} 
  = 
  \eps(\pi)
  S_{\pi}^{\ast} \calI_{k,r}^{-1},
  \\ 
  \calJ_{k,r}^{-1} S_{\pi}^{\ast} 
  = 
  \calJ_{k,r}^{-1} S_{\pi}^{\ast} \calJ_{k,r} \calJ_{k,r}^{-1} 
  = 
  \eps(\pi)
  \calJ_{k,r}^{-1} \calJ_{k,r} S_{\pi}^{\ast} \calJ_{k,r}^{-1} 
  = 
  \eps(\pi)
  S_{\pi}^{\ast} \calJ_{k,r}^{-1}.
\end{gather*}
 This completes the proof. 
\end{proof}

As a direct consequence of this theorem, the canonical isomorphisms and their inverses map $\bbF$-invariant sets onto $\bbF$-invariant sets. 
We will use the following important corollary for constructing $\bbF$-invariant bases.

\begin{corollary} \label{corollary:isomorphisminvariantbases}
 Let $T$ be an $n$-simplex, and $k,r \in \bbN_{0}$ with $r \geq 0$. Then:
 \begin{align*}
  \calQ \subseteq \calP_{r}\Lambda^{k}(T) \text{ is } \bbF \text{-invariant}
  &\iff 
  \calI_{k,r} \calQ \subseteq \mathring\calP_{r+k+1}^{-}\Lambda^{n-k}(T) \text{ is } \bbF \text{-invariant}
  \\
  \calQ \subseteq \calP_{r+1}^{-}\Lambda^{k}(T) \text{ is } \bbF \text{-invariant}
  &\iff 
  \calJ_{k,r} \calQ \subseteq \mathring\calP_{r+k+1}\Lambda^{n-k}(T) \text{ is } \bbF \text{-invariant}
 \end{align*}  
\end{corollary}

\section{Traces and Extension Operators} \label{sec:extension}

In this section we study the relation of simplicial symmetries 
with traces, extension operators, and geometric decompositions of bases. 
The traces of $\bbF$-invariant sets are $\bbF$-invariant again. 
Conversely, we discuss extension operators that preserve $\bbF$-invariant sets.
An important result is that an $\bbF$-invariant geometrically decomposed basis exists 
if and only if 
such bases exist for each component in the geometric decomposition. 

We first prove that taking traces preserves $\bbF$-invariance. 

\begin{lemma} \label{lemma:traces}
 Let $T$ be an $n$-dimensional simplex and let $F \subseteq T$ be a subsimplex. 
 If a finite set $\calQ \subseteq \bbF\Lambda^{k}(T)$ is $\bbF$-invariant, 
 then $\trace_{T,F} \calQ \subseteq \bbF\Lambda^{k}(F)$ is $\bbF$-invariant.
\end{lemma}

\begin{proof}
 Let $S_{} \in \Sym(T)$ such that $S_{}(F) = F$. 
 Let $\calQ = \left\{ \omega_{1}, \dots, \omega_{M} \right\}$, 
 where $M$ denotes the size of $\calQ$.
 Since $\calQ$ is $\bbF$-invariant,
 there exist 
 units $\chi_{1},\dots,\chi_{M} \in \bbF$
 and 
 a permutation $\tau \in \Perm(1:M)$
 such that $S_{}^{\ast} \omega_{i} = \chi_{i} \omega_{\tau(i)}$
 for $1 \leq i \leq M$.
 
 There exists $S_{F} \in \Sym(F)$
 which reorders the vertices of $F$ in the same way as $S$ does. 
 We observe $S \circ \imath(F,T) = \imath(F,T) \circ S_{F}$, 
 where $\imath(F,T) : F \rightarrow T$ is the natural inclusion.
 Hence 
 \begin{align*}
  S_{F}^{\ast} \trace_{T,F} \omega_{i}
  &=
  S_{F}^{\ast} \imath(F,T)^{\ast} \omega_{i}
  =
  \imath(F,T)^{\ast} S_{}^{\ast} \omega_{i}
  \\&=
  \trace_{T,F} S_{}^{\ast} \omega_{i}
  =
  \trace_{T,F} \chi_{i} \omega_{\tau(i)}
  =
  \chi_{i} \trace_{T,F} \omega_{\tau(i)}
  \in
  \chi_{i} \trace_{T,F} \calQ
  ,
 \end{align*}
 which had to be shown. \end{proof}

The idea of geometrically decomposed bases is central to finite element exterior calculus. 
Usually, geometrically decomposed bases are constructed explicitly via specific extension operators
\cite{AFWgeodecomp,licht2022basis}.
As a preparation, we introduce geometric decompositions on a slightly more abstract level
where, importantly, we already study $\bbF$-invariant sets.

Let $T$ be an $n$-simplex and let $k, r \in \bbN_{0}$ with $r > 0$.
Suppose that $\calQ\calP_{r}\Lambda^{k}(T)$ is a basis of $\calP_{r}\Lambda^{k}(T)$.
We call such a basis \emph{geometrically decomposed} if 
it is the disjoint union 
\begin{align} \label{math:generalgeometricdecomposition}
    \calQ\calP_{r}\Lambda^{k}(T)
    :=
    \bigcup_{ F \subseteq T } \calQ_{F}
\end{align}
where $F$ ranges over all the subsimplices of $T$,
and where $\calQ_{F}$ satisfies, on the one hand, 
that the trace from $T$ to $F$ maps $\calQ_{F}$ bijectively onto a basis of $\mathring\calP_{r}\Lambda^{k}(F)$,
and on the other hand, that $\trace_{T,G} \calQ_{F} = \{0\}$ whenever $G$ is a subsimplex of $T$ not containing $F$. 
We define \emph{geometrically decomposed} bases of $\calP^{-}_{r}\Lambda^{k}(T)$ completely analogously.

\begin{example}
 As we shall discuss in more details below,
 our notion of geometric decomposition is only a minor generalization 
 of earlier decompositions in the literature \cite{AFW1,AFWgeodecomp}.
 The bases $\calB\calP_{r}\Lambda^{k}(T)$ and $\calB\calP^{-}_{r}\Lambda^{k}(T)$ are geometrically decomposed. 
 Notably, not all of them are $\bbF$-invariant. 
 For further illustration, suppose that $T$ is a triangle. 
 The barycentric coordinates $\lambda^{T}_{0}, \lambda^{T}_{1}, \lambda^{T}_{2}$ are a geometrically decomposed basis of $\calP_{1}\Lambda^{0}(T)$,
 whereas the basis 
 $\lambda^{T}_{0}+\lambda^{T}_{1}, \lambda^{T}_{0}+\lambda^{T}_{2}, \lambda^{T}_{1}+\lambda^{T}_{2}$
 is not geometrically decomposed.
 Both bases, however, are $\bbF$-invariant. 
\end{example}

\begin{theorem} \label{theorem:geometricdecomposition:sullivan}
    Let $T$ be an $n$-simplex and let $k, r \in \bbN_{0}$ with $r > 0$.
    Let $\calQ\calP_{r}\Lambda^{k}(T)$ be a basis of $\calP_{r}\Lambda^{k}(T)$
    with geometric decomposition \eqref{math:generalgeometricdecomposition}.
    Then $\calQ_{T}$ is a basis of $\mathring\calP_{r}\Lambda^{k}(T)$.
    For any subsimplex $G$ of $T$, 
    a geometrically decomposed basis of $\calP_{r}\Lambda^{k}(G)$ is given by 
    \begin{align} \label{math:geometricdecomposition:sullivan}
    \calQ\calP_{r}\Lambda^{k}(G) = \bigcup_{ F \subseteq G } \trace_{T,G} \calQ_{F}
    .
    \end{align}
    If $\calQ\calP_{r}\Lambda^{k}(T)$ is $\bbF$-invariant, 
    then $\calQ_{T}$ and $\calQ\calP_{r}\Lambda^{k}(G)$ are $\bbF$-invariant. 
\end{theorem}

\begin{proof}
    Suppose that $\calQ\calP_{r}\Lambda^{k}(T)$ is a geometrically decomposed basis of $\calP_{r}\Lambda^{k}(T)$.
    That $\calQ_{T}$ is a basis of $\mathring\calP_{r}\Lambda^{k}(T)$ follows from definitions. 
    We show that $\calQ\calP_{r}\Lambda^{k}(G)$ is a basis of $\calP_{r}\Lambda^{k}(G)$
    when $G$ is any subsimplex of $T$. 
    
    By assumption, $\trace_{T,G} \calQ_{F} = \{0\}$ if $F$ is not a subsimplex of $G$. 
    If instead $F \subseteq G$, then by assumption, 
    $\trace_{T,F} : \calQ_{F} \rightarrow \trace_{T,F}\calQ_{F}$ 
    is a bijection and its image is basis of $\mathring\calP_{r}\Lambda^{k}(F)$.
    Because $\trace_{T,F} \calQ_{F} = \trace_{G,F} \trace_{T,G} \calQ_{F}$,
    the trace $\trace_{G,F} : \trace_{T,G} \calQ_{F} \rightarrow \trace_{T,F}\calQ_{F}$ 
    is a bijection too. 

    That $\calQ\calP_{r}\Lambda^{k}(G)$ spans $\calP_{r}\Lambda^{k}(G)$ is easily seen:
    \begin{align*}
     \calP_{r}\Lambda^{k}(G)
     =
     \trace_{T,G} \calP_{r}\Lambda^{k}(T)
     &=
     \sum_{ F \subseteq T } \trace_{T,G} \linhull\calQ_{F}
     \\
	 &=
     \sum_{ F \subseteq G } \trace_{T,G} \linhull\calQ_{F}
     =
     \sum_{ F \subseteq G } \linhull\trace_{T,G}\calQ_{F}
     .
    \end{align*}
    To show that \eqref{math:geometricdecomposition:sullivan} defines a linearly independent set,
	let $\omega \in \calP_{r}\Lambda^{k}(G)$
	be the sum of $\omega_F \in \linhull\trace_{T,G}\calQ_{F}$, $F \subseteq G$, not all zero.
	Then there exists $F$ of minimal dimension with $\omega_F \neq 0$,
	and thus $\trace_{G,F} \omega = \trace_{G,F} \omega_{F} \neq 0$. 
	In particular, $\omega \neq 0$. 
Lastly, that $\calQ\calP_{r}\Lambda^{k}(F)$ is geometrically decomposed follows from the observations above.
        
    Suppose that $\calQ\calP_{r}\Lambda^{k}(T)$ is $\bbF$-invariant. 
    By Lemma~\ref{lemma:traces}, the set of traces $\trace_{T,G}\calQ\calP_{r}\Lambda^{k}(T)$ is $\bbF$-invariant,
	and therefore its subset of non-zero traces is $\bbF$-invariant.
    Finally,
	$\calQ_{T}$ is $\bbF$-invariant
	since $\mathring\calP_{r}\Lambda^{k}(T) \cap \calQ\calP_{r}\Lambda^{k}(T) = \calQ_{T}$
    and $\mathring\calP_{r}\Lambda^{k}(T)$ is $\frakr$-invariant. 
\end{proof}

\begin{theorem} \label{theorem:geometricdecomposition:whitney}
    Let $T$ be an $n$-simplex and let $k, r \in \bbN_{0}$ with $r > 0$.
    Let $\calQ\calP^{-}_{r}\Lambda^{k}(T)$ be a basis of $\calP^{-}_{r}\Lambda^{k}(T)$
    with geometric decomposition analogous to \eqref{math:generalgeometricdecomposition}.
    Then $\calQ_{T}$ is a basis of $\mathring\calP^{-}_{r}\Lambda^{k}(T)$.
    For any subsimplex $G$ of $T$, 
    a geometrically decomposed basis of $\calP_{r}^{-}\Lambda^{k}(G)$ is given by 
    \begin{align}
    \calQ\calP^{-}_{r}\Lambda^{k}(G) = \bigcup_{ F \subseteq G } \trace_{T,G} \calQ_{F}
    .
    \end{align}
    If $\calQ\calP^{-}_{r}\Lambda^{k}(T)$ is $\bbF$-invariant, 
    then $\calQ_{T}$ and $\calQ\calP^{-}_{r}\Lambda^{k}(G)$ are $\bbF$-invariant. 
\end{theorem}

\begin{proof}
 This is completely analogous to the proof of Theorem~\ref{theorem:geometricdecomposition:sullivan}. 
\end{proof}

Up to now, we have studied properties of any geometrically decomposed basis 
and how this definition interacts with our notion of invariance. 
Most importantly, invariant geometrically decomposed bases 
give rise to invariant decomposed bases for certain subspaces and trace spaces. 
Shifting our focus to extension operators, more specific statements are possible. 
\\

Extension operators that facilitate geometric decompositions are widely used  
in finite element exterior calculus \cite{AFW1,AFWgeodecomp,licht2022basis}.
For our purpose, we utilize the extension operators given by Arnold, Falk, and Winther \cite{AFWgeodecomp}.
Their extension operators are described over spanning sets, not bases,
but this still yields well-defined linear mappings. 

Let $T$ be an $n$-dimensional simplex and let $F \subseteq T$ be an $m$-dimensional subsimplex,
and let $k, r \in \bbN_{0}$ with $r > 0$.
The extension operators for the $\calP_r^{-}\Lambda^{k}$-family of spaces, 
\begin{align}
 \ext^{k,r,-}_{F,T}
 :
 \mathring\calP^{-}_{r}\Lambda^{k}(F)
 \rightarrow
 \calP^{-}_{r}\Lambda^{k}(T)
 , 
\end{align}
are uniquely defined by setting 
\begin{align}
 \ext^{k,r,-}_{F,T} \lambda_{F}^{\alpha} \whitney^{F}_{\rho}
 =
 \lambda_{T}^{ \tilde\alpha } \whitney^{T}_{ \tilde\rho }
 ,
\end{align}
for all $\rho \in \Sigma_{0}(k,m)$ and $\alpha \in A(r-1,m)$, 
where $\tilde\rho = \imath(F,T) \circ \rho \in \Sigma_{0}(k,n)$, 
and where $\tilde\alpha \in A(r-1,n)$ is uniquely defined by requiring $\tilde\alpha \circ \imath(F,T) = \alpha$;
in particular, $\alpha$ is zero outside of $[\imath(F,T)]$.
This prescribes the extension operator 
over the spanning set $\calS\mathring\calP_r^{-}\Lambda^{k}(F)$ of the space $\mathring\calP_r^{-}\Lambda^{k}(F)$,
and one can show \cite[Section~7]{AFWgeodecomp} that this defines a linear operator.

The definition of the extension operators in the $\calP_r\Lambda^{k}$-family,
\begin{align}
 \ext^{k,r}_{F,T} :
 \mathring\calP_{r}\Lambda^{k}(F)
 \rightarrow
 \calP_{r}\Lambda^{k}(T)
 , 
\end{align}
is slightly more intricate. 
For any $\alpha \in A(r,n)$ and $\sigma \in \Sigma(k,n)$ we define 
\begin{align}
 \Psi_{i}^{\alpha,F,T}
 &:=
 \cartanlambda^{T}_{i} - \dfrac{\alpha(i)}{\lvert\alpha\rvert} \sum_{ j \in [\imath(F,T)] } \cartanlambda^{T}_{j}
 ,
 \quad i \in [0:n],
\\
 \Psi_{\sigma}^{\alpha,F,T}
 &:=
 \Psi_{\sigma(1)}^{\alpha,F,T}
 \wedge
 \dots
 \wedge
 \Psi_{\sigma(k)}^{\alpha,F,T}
 .
\end{align}
As described in \cite[Section~8]{AFWgeodecomp}, 
the extension operators are well-defined by setting 
\begin{align}
 \ext^{k,r}_{F,T}
 \lambda^{\alpha}_{F} \cartanlambda^{F}_{\sigma}
 =
 \lambda^{\tilde\alpha}_{T} \Psi_{\tilde\sigma}^{\tilde\alpha,F,T}
\end{align}
for all $\sigma \in \Sigma(1:k,0:m)$ and $\alpha \in A(r,m)$,
where $\tilde\sigma = \imath(F,T) \circ \sigma \in \Sigma(k,n)$, 
and where $\tilde\alpha \in A(r,n)$ is uniquely defined by requiring $\tilde\alpha \circ \imath(F,T) = \alpha$;
analogously to above, $\alpha$ is zero outside of $[\imath(F,T)]$.
This prescribes the extension operator 
over the spanning set $\calS\mathring\calP_r\Lambda^{k}(F)$ of the space $\mathring\calP_r\Lambda^{k}(F)$,
and it follows from \cite[Section~8]{AFWgeodecomp} that this defines a linear operator.

These operators are called \emph{extension} operators 
because they are right-inverses of the trace, 
\begin{align*}
 \trace_{T,F} \ext^{k,r,-}_{F,T} \omega &= \omega, 
 \quad 
 \omega \in \mathring\calP^{-}_{r}\Lambda^{k}(F),
 \\
 \trace_{T,F} \ext^{k,r}_{F,T} \omega &= \omega, 
 \quad 
 \omega \in \mathring\calP_{r}\Lambda^{k}(F).
\end{align*}
Moreover, whenever $G$ is another subsimplex of $T$,
then $F \subseteq G$ implies 
\begin{align*}
 \trace_{T,G} \ext^{k,r,-}_{F,T} \omega &= \ext^{k,r,-}_{F,G}, 
 \quad 
 \omega \in \mathring\calP^{-}_{r}\Lambda^{k}(F),
 \\
 \trace_{T,G} \ext^{k,r}_{F,T} \omega &= \ext^{k,r}_{F,G}, 
 \quad 
 \omega \in \mathring\calP_{r}\Lambda^{k}(F),
\end{align*}
whereas $F \nsubseteq G$ implies 
\begin{align*}
 \trace_{T,G} \ext^{k,r,-}_{F,T} \omega &= 0, 
 \quad 
 \omega \in \mathring\calP^{-}_{r}\Lambda^{k}(F),
 \\
 \trace_{T,G} \ext^{k,r}_{F,T} \omega &= 0, 
 \quad 
 \omega \in \mathring\calP_{r}\Lambda^{k}(F).
\end{align*}
We refer to prior publications \cite{AFWgeodecomp} for detailed discussion of these extension operators. 
The central result is the following decomposition.

\begin{theorem} \label{theorem:geodecomp}
    Let $T$ be an $n$-simplex and let $k, r \in \bbN_{0}$ with $r > 0$.
    Then 
    \begin{align*}
        \calP_{r}\Lambda^{k}(T)
        =
        \bigoplus_{F \subseteq T}
        \ext^{k,r}_{F,T} \mathring\calP_{r}\Lambda^{k}(F)
        ,
        \quad
        \calP_{r}^{-}\Lambda^{k}(T)
        =
        \bigoplus_{F \subseteq T}
        \ext^{k,r,-}_{F,T} \mathring\calP_{r}^{-}\Lambda^{k}(F)
        . 
    \end{align*}
\end{theorem}

\begin{remark}
 The geometric decomposition in Theorem~\ref{theorem:geodecomp} is fundamental to finite element theory
 and its importance can hardly be overstated: it is the geometric decomposition
 which enables the construction of localized bases.
 We refer to the literature \cite{AFWgeodecomp} for further background. 
\end{remark}

The above results on extension operators are known. 
Next we study how these extension operators interact with simplicial symmetries.

We introduce additional notation. 
Let $F \subseteq T$ be a subsimplex of a simplex $T$ and let $S \in \Sym(T)$.
Then $S_{} F$ is a subsimplex of $T$ of the same dimension as $F$,
and we have affine diffeomorphisms 
\begin{align}
 S_{\vert F}^{  } : F \rightarrow S_{} F,
 \quad 
 S_{\vert F}^{-1} : S_{} F \rightarrow F.
\end{align}

\begin{theorem} \label{theorem:extensionscommute}
    Let $k,r \in \bbN_{0}$ with $r > 0$.
    Let $T$ be a simplex of dimension $n$
    and 
    $F \subseteq T$ be a subsimplex of dimension $m$. 
    If $S \in \Sym(T)$,
    then  
    \begin{align} \label{math:extensionscommute}
        \ext^{k,r  }_{F,T} S_{\vert F}^{\ast} = S_{}^{\ast} \ext^{k,r  }_{SF,T},
        \quad 
        \ext^{k,r,-}_{F,T} S_{\vert F}^{\ast} = S_{}^{\ast} \ext^{k,r,-}_{SF,T}.
    \end{align}
\end{theorem}

\begin{proof}
    It suffices to prove both identities over the canonical spanning sets. 
    For every $S \in \Sym(T)$ and every subsimplex $F$ of $T$,
    we can decompose $S = S_{1} S_{2}$ for some $S_{1}, S_{2} \in \Sym(T)$ 
    where $S_{1|F} : F \rightarrow SF$ preserves the order of vertices and $S_{2}(F) = F$. 
    It suffices to consider simplicial symmetries $S$ belonging to one of the two special cases. 
    
    Let us suppose that $\pi \in \Perm(n)$ 
    such that $S_{\pi} \in \Sym(T)$ gives a mapping $S_{\pi|F} : F \rightarrow SF$
    that preserves the order of vertices. 
    We observe that $S_{\pi|F}^{\ast} \lambda^{SF}_{i} = \lambda^{F}_{i}$
    and thus $S_{\pi|F}^{\ast} \cartanlambda^{SF}_{i} = \cartanlambda^{F}_{i}$. 
    Hence, for any $\alpha \in A(r,n)$ and $\sigma \in \Sigma(k,n)$,
    \begin{align*}
     S_{\pi|F}^{\ast}
     \lambda_{SF}^{\alpha} \cartanlambda_{\sigma}^{SF}
     =
     \lambda_{ F}^{\alpha} \cartanlambda_{\sigma}^{ F}
     .
    \end{align*}
    We let $\alpha', \alpha'' \in A(r,n)$ be defined uniquely 
    by $\alpha' \circ \imath(SF,T) = \alpha$ and $\alpha'' \circ \imath( F,T) = \alpha$.
    We also abbreviate $\sigma' = \imath(SF,T) \circ \sigma$ and $\sigma'' = \imath( F,T) \circ \sigma$.
    Direct calculation verifies 
    \begin{align*}
     S_{\pi}^{\ast}
     \lambda_{T}^{\alpha'}
     \Psi_{\sigma'}^{\alpha',SF,T}
     =
     \lambda_{T}^{\alpha''}
     \Psi_{\sigma''}^{\alpha'', F,T}
     . 
    \end{align*}
    Similarly, 
    for any $\alpha \in A(r-1,n)$ and $\rho \in \Sigma_{0}(k,n)$, 
\begin{align*}
     S_{\pi|F}^{\ast}
     \lambda_{SF}^{\alpha} \whitney_{\rho}^{SF}
     =
     \lambda_{ F}^{\alpha} \whitney_{\rho}^{ F}
     .
    \end{align*}
    Letting $\alpha', \alpha'' \in A(r-1,n)$ be defined uniquely 
    by $\alpha' \circ \imath(SF,T) = \alpha$ and $\alpha'' \circ \imath( F,T) = \alpha$ 
    and abbreviating $\rho' = \imath(SF,T) \circ \rho$ and $\rho'' = \imath( F,T) \circ \rho$,
    we easily verify that 
    \begin{align*}
     S_{\pi}^{\ast}
     \lambda_{T}^{\alpha' } \whitney_{\rho' }^{T}
     =
     \lambda_{T}^{\alpha''} \whitney_{\rho''}^{T}
     . 
    \end{align*}
    This shows \eqref{math:extensionscommute} in the first special case. 
    
    We consider $S$ belonging to the second special case. 
    Let us suppose that $\pi \in \Perm(n)$ 
    such that $S_{\pi} \in \Sym(T)$ satisfies $S_{\pi}(F) = F$. 
    To approach the first identity, we prepare a few auxiliary results. 
    Let $\alpha \in A(r,n)$ and $\sigma \in \Sigma(k,n)$, 
    and let $i \in [0:n]$. 
    Since $S_\pi$ maps $F$ onto itself, 
    \begin{align*}
        S_{\pi}^{\ast} \Psi_{i}^{\alpha,F,T}
        &=
        S_{\pi}^{\ast} \cartanlambda^{T}_{i} - \lvert\alpha\rvert^{\inv} \alpha(i) S_{\pi}^{\ast} \sum_{ j \in [\imath(F,T)] } \cartanlambda^{T}_{j}
        \\&=
        \cartanlambda^{T}_{\pi(i)} - \lvert\alpha\rvert^{\inv} \alpha\pi^{\inv}(\pi(i)) \sum_{ j \in [\imath(F,T)] } \cartanlambda^{T}_{j}
        =
        \Psi_{\pi(i)}^{\alpha\pi^{\inv},F,T}
        .
    \end{align*}
    Letting $\widehat\sigma \in \Sigma(k,n)$ with $[\widehat\sigma] = [\pi\sigma]$, we find 
    \begin{align*}
        S_{\pi}^{\ast} \Psi_{\sigma}^{\alpha,F,T}
        &=
        S_{\pi}^{\ast} \Psi_{\sigma(1)}^{\alpha,F,T} \wedge \dots \wedge S_{\pi}^{\ast} \Psi_{\sigma(k)}^{\alpha,F,T}
        \\&=
        \Psi_{\pi\sigma(1)}^{\alpha\pi^{\inv},F,T} \wedge \dots \wedge \Psi_{\pi\sigma(k)}^{\alpha\pi^{\inv},F,T}
        \\&=
        \eps(\pi\sigma)
        \Psi_{\widehat\sigma(1)}^{\alpha\pi^{\inv},F,T} \wedge \dots \wedge \Psi_{\widehat\sigma(k)}^{\alpha\pi^{\inv},F,T}
        =
        \Psi_{\widehat\sigma}^{\alpha\pi^{\inv},F,T}
        .
    \end{align*}
With those preparations in place,
    let $\alpha \in A(r,m)$ and $\sigma \in \Sigma(k,m)$.
    Again, $\widehat\sigma \in \Sigma(k,m)$ with $[\widehat\sigma] = [\pi\sigma]$.
    Moreover, 
    we let $\tilde\sigma = \imath(F,T) \circ \sigma \in \Sigma(k,n)$
    and let $\tilde\alpha \in A(r-1,n)$ be defined by $\tilde\alpha \circ \imath(F,T) = \alpha$.
    We first verify that 
    \begin{gather*}
        S_{\pi}^{\ast} 
        \ext^{k,r}_{F,T}
        \lambda_{F}^{\alpha} \cartanlambda^{F}_{\sigma}
        = 
        S_{\pi}^{\ast} 
        \lambda_{T}^{ \tilde\alpha } \Psi^{\tilde\alpha,F,T}_{ \tilde\sigma }
        = 
        \lambda_{T}^{ \tilde\alpha \pi^{\inv} } 
        S_{\pi}^{\ast} 
        \Psi^{\tilde\alpha,F,T}_{ \tilde\sigma }
        \\ 
        \ext^{k,r}_{F,T} 
        S_{\pi|F}^{\ast} 
        \lambda_{F}^{\alpha} \cartanlambda^{F}_{\sigma}
        = 
        \eps(\pi\sigma)
        \ext^{k,r}_{F,T} 
        \lambda_{F}^{\alpha\pi^{\inv}} \cartanlambda^{F}_{\widehat\sigma}
        .
    \end{gather*}
    To show the first identity in \eqref{math:extensionscommute}, 
    we merely observe that $\alpha\pi^{\inv} = \tilde\alpha \pi^{\inv} \imath(F,T)$
    and that $[ \imath(F,T) \widehat\sigma ] = [\pi  \tilde\sigma ]$ with $\eps(\pi\sigma) = \eps(\pi\widetilde\sigma)$.
    The desired identity then follows from our auxiliary computations 
    and the definition of the extension operators.

    Lastly, we prove the second identity.
    Now let $\alpha \in A(r-1,m)$ and $\rho \in \Sigma_{0}(k,m)$.  
    Let $\widehat\rho \in \Sigma(k,m)$ with $[\widehat\rho] = [\pi\rho]$
    and $\tilde\rho = \imath(F,T) \circ \rho \in \Sigma_{0}(k,n)$.
    Similarly to above, 
    we let $\tilde\alpha \in A(r-1,n)$ be defined by $\tilde\alpha \circ \imath(F,T) = \alpha$. 
    We calculate 
    \begin{gather*}
        S_{\pi}^{\ast} 
        \ext^{k,r,-}_{F,T}
        \lambda_{F}^{\alpha} \whitney^{F}_{\rho}
        = 
        S_{\pi}^{\ast} 
        \lambda_{T}^{ \tilde\alpha } \whitney^{T}_{\tilde\rho}
        = 
        \lambda_{T}^{ \tilde\alpha \pi^{\inv} } 
        S_{\pi}^{\ast} 
        \whitney^{T}_{\tilde\rho}
        \\ 
        \ext^{k,r,-}_{F,T} 
        S_{\pi|F}^{\ast} 
        \lambda_{F}^{\alpha} \whitney^{F}_{\rho}
        = 
        \eps(\pi\rho)
        \ext^{k,r,-}_{F,T} 
        \lambda_{F}^{\alpha\pi^{\inv}} \whitney^{F}_{\widehat\rho}
        .
    \end{gather*}
    To show the second identity in \eqref{math:extensionscommute}, 
    we see $\alpha\pi^{\inv} = \tilde\alpha \pi^{\inv} \imath(F,T)$
    and $[ \imath(F,T) \widehat\rho ] = [\pi  \tilde\rho ]$ with $\eps(\pi\rho) = \eps(\pi\widetilde\rho)$.
    The desired identity follows via the definition of the extension operators.     
\end{proof}

We now work along the following idea: 
if a basis allows for a geometric decomposition corresponding to Theorem~\ref{theorem:geodecomp}, 
then this basis is $\bbF$-invariant under the provision that the components in the geometric decomposition  
satisfy certain invariance properties.
This is formalized in the following two theorems,
which strengthen Theorems~\ref{theorem:geometricdecomposition:sullivan}~and~\ref{theorem:geometricdecomposition:whitney}. 
An important consequence is this:
$\bbF$-invariant bases for the components in that geometric decomposition 
yield $\bbF$-invariant bases for the entire finite element space over the simplex. 

\begin{theorem} \label{theorem:geodecompinvariant:fullspace}
 Let $k, r \in \bbN_{0}$ with $r > 0$ and let $T$ be an $n$-simplex.
Assume that $\calQ\mathring\calP_{r}\Lambda^{k}(F)$ is a basis for $\mathring\calP_{r}\Lambda^{k}(F)$
 for each subsimplex $F \subseteq T$,
 and define  
 \begin{align*}
  \calQ\calP_r\Lambda^{k}(T)
  :=
  \bigcup_{ F \subseteq T } \ext^{k,r}_{F,T} \calQ\mathring\calP_r\Lambda^{k}(F)
  .
 \end{align*}
 Then $\calQ\calP_r\Lambda^{k}(T)$ is a basis of $\calP_r\Lambda^{k}(T)$.
 The following statements are equivalent:
 \begin{itemize}
  \item 
  $\calQ\calP_r\Lambda^{k}(T)$ is $\bbF$-invariant. 
  \item 
  For each subsimplex $F$ of $T$ 
  the set $\calQ\mathring\calP_r\Lambda^{k}(F)$ is $\bbF$-invariant 
  and for each $S \in \Sym(T)$ that preserves the relative order of vertices of $F$
  we have 
  \begin{align*}
   S^{\ast} \ext^{k,r}_{SF,T} \calQ\mathring\calP_r\Lambda^{k}(SF)
   =
   \ext^{k,r}_{F,T} \calQ\mathring\calP_r\Lambda^{k}(F)
   .
  \end{align*}
 \end{itemize}
\end{theorem}

\begin{proof}
 For any subsimplex $F \subseteq T$,
 since $\calQ\mathring\calP_r\Lambda^{k}(F)$ is a basis for $\mathring\calP_r\Lambda^{k}(F)$
 we see that 
 $\ext^{k,r}_{F,T} \calQ\mathring\calP_r\Lambda^{k}(F)$ is a basis for $\ext^{k,r}_{F,T} \mathring\calP_r\Lambda^{k}(F)$,
 via Theorem~\ref{theorem:extensionscommute}. 
 By Theorem~\ref{theorem:geodecomp} then, 
 it is clear that $\calQ\calP_{r}\Lambda^{k}(T)$ is a basis of $\calP_{r}\Lambda^{k}(T)$. 
 
 Assume that $\calQ\calP_{r}\Lambda^{k}(T)$ is $\bbF$-invariant. 
 Let $F \subseteq T$ be a subsimplex and let $S \in \Sym(F)$. 
 There exists $\widehat S \in \Sym(T)$ such that $S = \widehat S_{|F}$. 
 The pullback along $\widehat S$ preserves the space $\ext^{k,r}_{F,T} \mathring\calP_{r}\Lambda^{k}(F)$ since 
 Theorem~\ref{theorem:extensionscommute} implies 
 \begin{align*}
  \widehat S^{\ast} \ext^{k,r}_{F,T} \mathring\calP_{r}\Lambda^{k}(F) 
  =
  \ext^{k,r}_{F,T} S^{\ast} \mathring\calP_{r}\Lambda^{k}(F) 
  =
  \ext^{k,r}_{F,T} \mathring\calP_{r}\Lambda^{k}(F) 
  .
 \end{align*}
 Because $\calQ\calP_{r}\Lambda^{k}(T)$ is $\bbF$-invariant,
for every $\omega_{1} \in \calQ\mathring\calP_r\Lambda^{k}(F)$
 there exists $\omega_{2} \in \calQ\mathring\calP_r\Lambda^{k}(F)$
 and a complex unit $\chi \in \bbF$ such that 
 $\widehat S^{\ast} \ext^{k,r}_{F,T} \omega_{1} = \chi \ext^{k,r}_{F,T} \omega_{2}$.
 Using Theorem~\ref{theorem:extensionscommute} again, we note that 
 \begin{align*}
  S^{\ast} \omega_{1}
  = 
  \trace_{T,F} \ext^{k,r}_{F,T} S^{\ast} \omega_{1}
  = 
  \trace_{T,F} \widehat S^{\ast} \ext^{k,r}_{F,T} \omega_{1}
  = 
  \chi 
  \trace_{T,F} \ext^{k,r}_{F,T} \omega_{2}
  = 
  \chi 
  \omega_{2}
  .
 \end{align*}
 Hence, by definition, $\calQ\mathring\calP_r\Lambda^{k}(F)$ is $\bbF$-invariant. 
 
 Let $F \subseteq T$ be a subsimplex and let $S \in \Sym(T)$ 
 be such that the mappings
 \begin{align}
  S_{\vert F}^{  } : F \rightarrow S_{} F,
  \quad 
  S_{\vert F}^{-1} : S_{} F \rightarrow F
 \end{align}
 preserve the relative order of the vertices. 
 By Theorem~\ref{theorem:extensionscommute}, we have 
 \begin{align*}
  S^{\ast}
  \ext^{k,r}_{SF,T} \mathring\calP_{r}\Lambda^{k}(SF)
  =
  \ext^{k,r}_{ F,T} S_{\vert F}^{\ast} \mathring\calP_{r}\Lambda^{k}(SF)
  =
  \ext^{k,r}_{ F,T} \mathring\calP_{r}\Lambda^{k}(F)
.
 \end{align*}
 The same argument can be applied to the inverse of $S$;
 it follows that we have an isomorphism \begin{align*}
  S^{\ast} : 
  \ext^{k,r}_{SF,T} \mathring\calP_{r}\Lambda^{k}(SF)
  \rightarrow 
  \ext^{k,r}_{F,T} \mathring\calP_{r}\Lambda^{k}(F).
 \end{align*}
 As $\calQ\calP_{r}\Lambda^{k}(T)$ is $\bbF$-invariant,
 $S^{\ast}$ maps $\ext^{k,r}_{SF,T} \calQ\mathring\calP_{r}\Lambda^{k}(SF)$
 bijectively onto 
$\ext^{k,r}_{F,T} \calQ\mathring\calP_{r}\Lambda^{k}(F)$
 up to multiplication by units of $\bbF$. 
 But since $S$ preserves the relative ordering of the vertices of $F$,
 a direct calculation shows that those units must equal one. 
 Hence we have a bijection 
 \begin{align*}
  S^{\ast} : 
  \ext^{k,r}_{SF,T} \calQ\mathring\calP_{r}\Lambda^{k}(SF)
  \rightarrow 
  \ext^{k,r}_{F,T} \calQ\mathring\calP_{r}\Lambda^{k}(F).
 \end{align*}
 Thus, we have shown that the first statement of the theorem implies the second statement.

 It remains to show the converse implication,
 so let us assume that the second statement is true.
 Let $S \in \Sym(T)$.
 There exist $S_{1}, S_{2} \in \Sym(T)$ such that $S = S_{1} S_{2}$, 
 we have $S_{1}(F) = SF$ and $S_{2}(F) = F$, and $S_{1 \vert F} : F \rightarrow SF$ preserves the order of vertices.

 Let $\omega \in \calQ\calP_{r}\Lambda^{k}(T)$. 
 There exists a subsimplex $F$ of $T$ and $\omega_{0} \in \calQ\mathring\calP_{r}\Lambda^{k}(SF)$
 such that $\omega = \ext^{k,r}_{SF,T} \omega_{0}$. 
Note that by assumption, we have a bijection 
 \begin{align*}
  S_{1}^{\ast} : 
  \ext^{k,r}_{SF,T} \calQ\mathring\calP_{r}\Lambda^{k}(SF)
  \rightarrow 
  \ext^{k,r}_{F,T} \calQ\mathring\calP_{r}\Lambda^{k}(F).
 \end{align*}
 Hence,
there exists $\omega_{1} \in \calQ\mathring\calP_{r}\Lambda^{k}(F)$
 such that $S_{1}^{\ast} \ext^{k,r}_{SF,T} \omega_{0} = \ext^{k,r}_{F,T} \omega_{1}$. 
Furthermore, since $\calQ\mathring\calP_{r}\Lambda^{k}(F)$ is assumed to be $\bbF$-invariant, 
 there exist a complex unit $\chi \in \bbF$ 
 and $\omega_2 \in \calQ\mathring\calP_{r}\Lambda^{k}(F)$
 such that $S_{2 \vert F}^{\ast} \omega_{1} = \chi \omega_{2}$. 
 Thus 
 \begin{align*}
  S^{\ast} \omega 
  = 
S_{2}^{\ast} S_{1}^{\ast} \ext^{k,r}_{SF,T} \omega_{0}
  = 
  S_{2}^{\ast} \ext^{k,r}_{F,T} \omega_{1}
  = 
  \ext^{k,r}_{F,T} S_{2 \vert F}^{\ast} \omega_{1}
  = 
  \chi \ext^{k,r}_{F,T} \omega_{2}
  .
 \end{align*}
 As a consequence,
 $\calQ\calP_{r}\Lambda^{k}(T)$ is $\bbF$-invariant. 
\end{proof}

\begin{theorem} \label{theorem:geodecompinvariant:whitney}
 Let $k, r \in \bbN_{0}$ with $r > 0$
 and let $T$ be an $n$-simplex.
Assume that $\calQ\mathring\calP_{r}^{-}\Lambda^{k}(F)$ is a basis for $\mathring\calP_{r}^{-}\Lambda^{k}(F)$
 for each subsimplex $F \subseteq T$,
 and define  
 \begin{align*}
  \calQ\calP_r^{-}\Lambda^{k}(T)
  :=
  \bigcup_{ F \subseteq T } \ext^{k,r,-}_{F,T} \calQ\mathring\calP_r^{-}\Lambda^{k}(F)
  .
 \end{align*}
 Then $\calQ\calP_r^{-}\Lambda^{k}(T)$ is a basis of $\calP_r^{-}\Lambda^{k}(T)$.
 The following statements are equivalent:
 \begin{itemize}
  \item 
  $\calQ\calP_r^{-}\Lambda^{k}(T)$ is $\bbF$-invariant. 
  \item 
  For each subsimplex $F$ of $T$ 
  the set $\calQ\mathring\calP_r^{-}\Lambda^{k}(F)$ is $\bbF$-invariant 
  and for each $S \in \Sym(T)$ that preserves the relative order of vertices of $F$
  we have 
  \begin{align*}
   S^{\ast} \ext^{k,r,-}_{SF,T} \calQ\mathring\calP_r^{-}\Lambda^{k}(SF)
   =
   \ext^{k,r,-}_{F,T} \calQ\mathring\calP_r^{-}\Lambda^{k}(F)
   .
  \end{align*}
 \end{itemize}
\end{theorem}

\begin{proof}
 This is proven completely analogously to the preceding theorem.
\end{proof}

A geometrically decomposed basis enjoys $\bbF$-invariance 
if each component in the geometric decomposition is already $\bbF$-invariant
and if, additionally, the components associated to subsimplices of the same dimension
are ``reindexings'' of each other. 
That second additional condition is a mere technicality:
if we know an $\bbF$-invariant basis for the spaces of vanishing trace over, say, any $m$-dimensional reference simplex,
then pullback along vertex-order preserving affine diffeomorphisms immediately defines such components associated to every $m$-dimensional subsimplex.
Hence the only question of technical interest remaining for constructing geometrically decomposed $\bbF$-invariant bases
is whether $\bbF$-invariant bases are known for the corresponding spaces with vanishing boundary traces on all subsimplices.

\section{Recursive Basis Construction} \label{sec:recursion}

This final section describes the recursive construction of geometrically decomposed bases 
for simplicial higher-degree finite element spaces  
in finite element exterior calculus. 
We combine the results of the previous sections.
\\

We commence the construction as follows. 
First, 
for every simplex $T$ we fix a basis $\calA\calP_{0}\Lambda^{k}(T)$
for the lowest-degree space $\calP_{0}\Lambda^{k}(T)$.
In principle, any arbitrary choice of bases can serve as the base case in our recursive construction.
But some specific choices of bases, given further below, lead to $\bbF$-invariant higher-degree bases. 

Moreover, recall that the empty set is a basis for the trivial vector space. 
To simplify some technical arguments, 
we therefore fix the empty bases whenever the corresponding vector space is trivial. 
Specifically, 
we let $\calA\calP_{r}\Lambda^{k}(T) = \emptyset$ when $r < 0$
and    $\calA\calP_{r}^{-}\Lambda^{k}(T) = \emptyset$ when $r \leq 0$.
Similarly, we set
$\calA\mathring\calP_{r}^{-}\Lambda^{k}(T) = \emptyset$ when $r < n-k+1$,
and we set 
$\calA\mathring\calP_{r}^{ }\Lambda^{k}(T) = \emptyset$ either when $k < n$ and $r < n-k+1$ or when $k=n$ and $r < 0$. 
\\

Second, 
recall the canonical isomorphisms of Section~\ref{sec:canonicalisomorphism},
\begin{gather*}
 \calI_{k,r} : \calP_{r}\Lambda^{k}(T) \rightarrow \mathring\calP^{-}_{r+k+1}\Lambda^{n-k}(T)
 ,
 \quad 
 \calJ_{k,r} : \calP_{r+1}^{-}\Lambda^{k}(T) \rightarrow \mathring\calP_{r+k+1}\Lambda^{n-k}(T)
 .
\end{gather*}
These isomorphisms are defined for $r \geq 0$ and map bases to bases. 
Shifting indices, we construct bases for the finite element spaces with homogeneous boundary traces as follows. When $r \geq n-k$, we set 
\begin{gather*}
 \calA\mathring\calP^-_{r+1}\Lambda^{k}(T)
 :=
 \calI_{n-k,r-n+k}
 \calA\calP_{r-n+k}\Lambda^{n-k}(T), 
 \\
 \calA\mathring\calP_{r+1}\Lambda^{k}(T)
 :=
 \calJ_{n-k,r-n+k}
 \calA\calP^{-}_{r-n+k+1}\Lambda^{n-k}(T)
 ,
\end{gather*}
provided that the bases $\calA\calP_{r-n+k}\Lambda^{n-k}(T)$
and $\calA\calP^{-}_{r-n+k+1}\Lambda^{n-k}(T)$
have already been constructed. 
As seen in Corollary~\ref{corollary:isomorphisminvariantbases}, 
the bases $\calA\mathring\calP^{-}_{r+1}\Lambda^{k}(T)$ and $\calA\mathring\calP_{r+1}\Lambda^{k}(T)$
are $\bbF$-invariant 
if and only if 
the bases $\calA\calP_{r-n+k}\Lambda^{n-k}(T)$ and $\calA\calP^{-}_{r-n+k+1}\Lambda^{n-k}(T)$
are $\bbF$-invariant,
respectively.
\\

Third, 
taking into account the geometric decompositions 
\begin{gather*}
 \calP_{r}\Lambda^{k}(T)
 =
 \bigoplus_{F \subseteq T}
 \ext^{k,r}_{F,T} \mathring\calP_{r}\Lambda^{k}(F)
 ,
 \\
 \calP_{r}^{-}\Lambda^{k}(T)
 =
 \bigoplus_{F \subseteq T}
 \ext^{k,r,-}_{F,T} \mathring\calP^{-}_{r}\Lambda^{k}(F)
 , 
\end{gather*}
we define bases  
\begin{subequations}
\begin{gather}
 \calA\calP_{r}\Lambda^{k}(T)
 :=
 \bigcup_{F \subseteq T}
 \ext^{k,r}_{F,T} \calA\mathring\calP_{r}\Lambda^{k}(F)
 ,
 \\
 \calA\calP_{r}^{-}\Lambda^{k}(T)
 :=
 \bigcup_{F \subseteq T}
 \ext^{k,r,-}_{F,T} \calA\mathring\calP^{-}_{r}\Lambda^{k}(F)
 . 
\end{gather}
\end{subequations}
provided that the corresponding bases for the finite element spaces 
with boundary conditions have already been constructed. 

The basis $\calA\calP_{r}\Lambda^{k}(T)$ is $\bbF$-invariant
if the bases $\calA\mathring\calP_{r}\Lambda^{k}(F)$ for every $F$
satisfy the conditions of Theorem~\ref{theorem:geodecompinvariant:fullspace}.
Analogously, 
the basis $\calA\calP_{r}^{-}\Lambda^{k}(T)$ is $\bbF$-invariant
if the bases $\calA\mathring\calP^{-}_{r}\Lambda^{k}(F)$ for every $F$
satisfy the conditions of Theorem~\ref{theorem:geodecompinvariant:whitney}.
In both cases, $\bbF$-invariant bases on the lower-dimensional subsimplices are easily obtained 
if we know $\bbF$-invariant bases on any simplices of the corresponding dimensions.
\\

It is now clear how to recursively construct bases 
for finite element spaces of arbitrary polynomial degree:
starting from any given bases for constant differential forms, 
we construct bases of higher polynomial degree, 
using the canonical isomorphisms or the geometric decomposition. 
In particular, 
if the initial zero-degree finite element bases are $\bbF$-invariant,
then the recursively constructed higher-degree finite element bases 
will be $\bbF$-invariant as well. 

We develop this idea when the simplex $T$ has practically relevant low dimension. 
As explained above, 
our construction of higher-degree $\bbF$-invariant bases relies on 
$\bbF$-invariant bases for the zero-degree spaces,
which constitute the base case of the recursion. 
For the remainder of this section,
we therefore make the following additional assumptions:  
\begin{itemize}
 \item 
 On any simplex $T$, 
 the basis $\calA\calP_{0}\Lambda^{0}(T)$
 consists of the constant function with pointwise value $1$. 
 \item 
 On any $n$-dimensional simplex $T$,
 the bases $\calA\calP_{0}\Lambda^{n}(T)$ and $\calA\mathring\calP_{0}\Lambda^{n}(T)$
 consist of the volume form $\vol_T$. 
 \item 
 If $T$ is a triangle,
 then $\calA\calP_{0}\Lambda^{1}(T)$
 is the $\bbC$-invariant basis of Lemma~\ref{lemma:invariantconstant:1forms2D}. 
 \item 
 If $T$ is a tetrahedron,
 then $\calA\calP_{0}\Lambda^{1}(T)$ and $\calA\calP_{0}\Lambda^{2}(T)$
 are the $\bbR$-invariant bases of Lemmas~\ref{lemma:invariantconstant:1forms3D}~and~\ref{lemma:invariantconstant:2forms3D}.
 \item 
 If $T$ is a $4$-simplex, 
 then $\calA\calP_{0}\Lambda^{2}(T)$
 is the $\bbC$-invariant basis of Lemma~\ref{lemma:invariantconstant:2forms4D}.
\end{itemize}
This stipulation leads to the following results.
We begin with finite element bases for higher-degree scalar and volume forms to illustrate the basic ideas.

\begin{theorem} \label{theorem:recursiveconstruction:scalarvolume}
 Let $T$ be an $n$-dimensional simplex, and let $r \in \bbN_{0}$. Then the bases 
 \begin{gather*}
   \calA\calP_{r}\Lambda^{0}(T),
   \quad 
   \calA\calP_{r}\Lambda^{n}(T),
   \quad 
   \calA\mathring\calP_{r}\Lambda^{0}(T),
   \quad
   \calA\mathring\calP_{r}\Lambda^{n}(T),
   \\
   \calA\calP_{r}^{-}\Lambda^{0}(T),
   \quad
   \calA\calP_{r}^{-}\Lambda^{n}(T), 
   \quad
   \calA\mathring\calP_{r}^{-}\Lambda^{0}(T),
   \quad
   \calA\mathring\calP_{r}^{-}\Lambda^{n}(T)
 \end{gather*}
 are $\bbR$-invariant. 
\end{theorem}

\begin{proof}
  We perform nested induction over the dimension $n$ and the polynomial degree $r$.
  The base case are simplices of dimension $n = 0$, that is, vertices. 
  
  So suppose that $n=0$. The statement holds when $r=0$,
  where the bases $\calA\calP_{0}\Lambda^{0}(T)$ and $\calA\mathring\calP_{0}\Lambda^{0}(T)$ are $\bbR$-invariant,
  and the sets $\calA\calP^{-}_{0}\Lambda^{0}(T)$ and $\calA\mathring\calP^{-}_{0}\Lambda^{0}(T)$ are empty.
Suppose the statement is true for some polynomial degree $r \geq 0$. 
  Note that $\calA\calP_{r}\Lambda^{0}(T)$ cannot be empty. 
  The first canonical isomorphism maps $\calA\calP_{r}\Lambda^{0}(T)$ to $\calA\mathring\calP^{-}_{r+1}\Lambda^{0}(T)$,
  and the geometric decomposition trivially leads to $\calA\calP^{-}_{r+1}\Lambda^{0}(T)$.
  The second canonical isomorphism maps $\calA\calP^{-}_{r+1}\Lambda^{0}(T)$ to $\calA\mathring\calP_{r+1}\Lambda^{0}(T)$,
  and the geometric decomposition now leads to $\calA\calP_{r+1}\Lambda^{0}(T)$.
By construction, the statement thus holds for the polynomial degree $r+1$. 
  The principle of induction implies the statement for any polynomial degree when $n=0$.

  Suppose that $T$ is an arbitrary but fixed simplex of dimension $n$
  and that the statement is true for all polynomial degrees on simplices up to dimension $n-1$.
  We then use induction over the polynomial degree $r$. 
  
We first consider the case $r=0$.
  We have $\bbR$-invariant bases $\calA\calP_{0}\Lambda^{0}(T)$ and $\calA\calP_{0}\Lambda^{n}(T)$.
  The set $\calA\mathring\calP_{0}\Lambda^{0}(T)$ is empty
  and the basis $\calA\mathring\calP_{0}\Lambda^{n}(T)$ is $\bbR$-invariant. 
  The sets $\calA\calP^{-}_{0}\Lambda^{0}(T)$, $\calA\calP^{-}_{0}\Lambda^{n}(T)$, 
  $\calA\mathring\calP^{-}_{0}\Lambda^{n}(T)$ and $\calA\mathring\calP^{-}_{0}\Lambda^{n}(T)$ are empty.
  
Next, 
  we assume the statement holds up to polynomial degree $r-1 \geq 0$ and prove it for degree $r$.
Using that $r \geq 1$, the first canonical isomorphism maps $\calA\calP^{ }_{r  -1}\Lambda^{0}(T)$ onto $\calA\mathring\calP^{-}_{r}\Lambda^{n}(T)$.
  We immediately get the basis $\calA\calP^{-}_{r}\Lambda^{n}(T)$ from the geometric decomposition.
When $r \leq n$, then $\calA\mathring\calP^{-}_{r}\Lambda^{0}(T)$ is empty,
  and otherwise, the first canonical isomorphism maps $\calA\calP^{ }_{r-n-1}\Lambda^{n}(T)$ onto $\calA\mathring\calP^{-}_{r}\Lambda^{0}(T)$.
  The geometric decomposition provides the basis $\calA\calP^{-}_{r}\Lambda^{0}(T)$.
Using that $r \geq 1$, the second canonical isomorphism maps $\calA\calP^{-}_{r    }\Lambda^{0}(T)$ onto $\calA\mathring\calP^{ }_{r}\Lambda^{n}(T)$.
  We immediately get the basis $\calA\calP^{ }_{r}\Lambda^{n}(T)$.
When $r \leq n$, then $\calA\mathring\calP^{-}_{r}\Lambda^{0}(T)$ is empty,
  and otherwise, the second canonical isomorphism maps $\calA\calP^{-}_{r-n  }\Lambda^{n}(T)$ onto $\calA\mathring\calP^{ }_{r}\Lambda^{0}(T)$.
  The geometric decomposition provides the basis $\calA\calP^{ }_{r}\Lambda^{0}(T)$.
  
  We can now apply the principle of induction over the polynomial degree $r$.
Lastly, we have completed the induction over the dimension of the simplex. 
  We conclude that the statement holds for all polynomial degrees $r$ and over simplices $T$ of any dimension $n$. 
\end{proof}

Similar techniques lead to the main results of this section, 
which now follow.

\begin{theorem} \label{theorem:recursiveconstruction:2D}
 Let $T$ be a triangle, and let $r \in \bbN_{0}$. Then the bases 
 \begin{align*}
  \calA\calP_{r}\Lambda^{1}(T), \quad \calA\calP_{r}^{-}\Lambda^{1}(T), 
  \quad 
  \calA\mathring\calP_{r}\Lambda^{1}(T), \quad \calA\mathring\calP_{r}^{-}\Lambda^{1}(T)
 \end{align*}
 are $\bbC$-invariant. Furthermore:
\begin{itemize}
\item The basis $\calA\calP_{r}\Lambda^{1}(T)$ is $\bbR$-invariant if and only if $r \notin 3 \bbN_{0}$.
  \item The basis $\calA\calP_{r}^{-}\Lambda^{1}(T)$ is $\bbR$-invariant if and only if $r \notin 3 \bbN_{0} + 2$. 
\item The basis $\calA\mathring\calP_{r}\Lambda^{1}(T)$ is $\bbR$-invariant if and only if $r \notin 3 \bbN_{0} + 3$.
  \item The basis $\calA\mathring\calP_{r}^{-}\Lambda^{1}(T)$ is $\bbR$-invariant if and only if $r \notin 3 \bbN_{0} + 2$. 
\end{itemize}
\end{theorem}

\begin{proof}
Let $E_{0}, E_{1}, E_{2}$ denote the edges of $T$. 
 When $r \geq 1$, the geometric decomposition establishes
 \begin{align*}
  \calA\calP_{r}\Lambda^{1}(T)
  &= 
  \calA\mathring\calP_{r}\Lambda^{1}(T)
  \cup
  \bigcup_{i=0}^{2}
  \ext^{k,r}_{E_{i},T}
  \calA\mathring\calP_{r}\Lambda^{1}(E_{i})
  , 
  \\
  \calA\calP_{r}^{-}\Lambda^{1}(T)
  &= 
  \calA\mathring\calP_{r}^{-}\Lambda^{1}(T)
  \cup
  \bigcup_{i=0}^{2}
  \ext^{k,r,-}_{E_{i},T}
  \calA\mathring\calP_{r}^{-}\Lambda^{1}(E_{i})
  .
 \end{align*}
 We have seen that the bases 
 $\calA\mathring\calP_{r}\Lambda^{1}(E_{i})$ and $\calA\mathring\calP_{r}^{-}\Lambda^{1}(E_{i})$
 are $\bbR$-invariant, and hence $\bbC$-invariant. 
 Moreover, the canonical isomorphisms lead to \begin{align*}
  \calA\mathring\calP_{r+2}^{-}\Lambda^{1}(T)
  =
  \calI_{1,r}
  \calA\calP_{r}\Lambda^{1}(T)
  ,
  \quad 
  \calA\mathring\calP_{r+2}\Lambda^{1}(T)
  =
  \calJ_{1,r}
  \calA\calP_{r+1}^{-}\Lambda^{1}(T)
  .
 \end{align*}
 We recall from Lemma~\ref{lemma:invariantconstant:1forms2D} that the basis $\calA\calP_{0}\Lambda^{1}(T)$ is $\bbC$-invariant but not $\bbR$-invariant. 
 
 We prove the theorem using induction over the polynomial degree, 
 beginning with the polynomial degrees $r \leq 3$.
The basis $\calA\calP_{0}\Lambda^{1}(T)$ is $\bbC$-invariant but not $\bbR$-invariant,
 and the bases $\calA\calP^{-}_{0}\Lambda^{1}(T)$, $\calA\mathring\calP_{0}\Lambda^{1}(T)$, and $\calA\mathring\calP^{-}_{0}\Lambda^{1}(T)$ are empty. 
Via the geometric decompositions, the bases $\calA\calP_{1}\Lambda^{1}(T)$ and $\calA\calP_{1}^{-}\Lambda^{1}(T)$ are $\bbR$-invariant,
 and the bases $\calA\mathring\calP_{1}\Lambda^{1}(T)$ and $\calA\mathring\calP_{1}^{-}\Lambda^{1}(T)$ are empty. 
Via the canonical isomorphisms and the geometric decompositions,
 the bases $\calA\mathring\calP_{2}\Lambda^{1}(T)$ and $\calA\calP_{2}\Lambda^{1}(T)$ are $\bbR$-invariant,
 while the bases $\calA\mathring\calP^{-}_{2}\Lambda^{1}(T)$ and $\calA\calP^{-}_{2}\Lambda^{1}(T)$ are $\bbC$-invariant but not $\bbR$-invariant. 
Via the canonical isomorphisms and the geometric decompositions again,
 $\calA\mathring\calP_{3}\Lambda^{1}(T)$ and $\calA\calP_{3}\Lambda^{1}(T)$ are $\bbC$-invariant but not $\bbR$-invariant,
 while the bases $\calA\mathring\calP^{-}_{3}\Lambda^{1}(T)$ and $\calA\calP^{-}_{3}\Lambda^{1}(T)$ are $\bbR$-invariant. 
This establishes the statement for $r \leq 3$.

 Next, we assume that $r \in \bbN$ with $r > 3$ and that the claim is true for polynomial degrees $s < r$. 
 Again, we use the recursive construction.
 All the following bases are $\bbC$-invariant by construction, and we only need to establish whether they are $\bbR$-invariant. 
On the one hand,
 $\calA\calP_{r}^{-}\Lambda^{1}(T)$ is $\bbR$-invariant
 if and only if 
 $\calA\mathring\calP_{r}^{-}\Lambda^{1}(T)$ is $\bbR$-invariant,
 which is the case if and only if 
 $\calA\calP_{r-2}\Lambda^{1}(T)$ is $\bbR$-invariant.
 By the induction assumption, 
 this is the case if and only if $r-2 \notin 3\bbN_{0}$, that is, $r \notin 3 \bbN_{0} + 2$.
On the other hand,
 $\calA\calP_{r}\Lambda^{1}(T)$ is $\bbR$-invariant
 if and only if 
 $\calA\mathring\calP_{r}\Lambda^{1}(T)$ is $\bbR$-invariant,
 which is the case if and only if 
 $\calA\calP^{-}_{r-1}\Lambda^{1}(T)$ is $\bbR$-invariant.
 By the induction assumption, 
 this is the case if and only if $r-1 \notin 3\bbN_{0}+2$, that is, $r \notin 3 \bbN_{0}+3$.
This completes the induction argument and the proof.
\end{proof}

\begin{remark}
 The basic idea of the above proof is that the bases 
 $\calA\calP_{r}\Lambda^{k}(T)$ and $\calA\calP_{r}^{-}\Lambda^{k}(T)$
 are not only $\bbC$-invariant but also $\bbR$-invariant,
 unless the recursive construction leads back to the basis $\calA\calP_{0}\Lambda^{1}(T)$, 
 which is not $\bbR$-invariant. 
 Loosely speaking,
 we trace the contributions of $\calA\calP_{0}\Lambda^{1}(T)$ throughout the recursion.
\end{remark}

\begin{example}
 Let us restate this result in the language of vector analysis:
 On any triangle $T$ and any polynomial degree, 
 the Raviart-Thomas space $\RT_{r}(T)$ and the Brezzi-Douglas-Marini space $\BDM_{r}(T)$
 have geometrically decomposed $\bbC$-invariant bases. 
 In addition to that,
 that basis of $\BDM_{r}(T)$ is $\bbR$-invariant if $r \notin 3 \bbN_{0}$,
 and 
 that basis of $\RT_{r}(T)$ is $\bbR$-invariant if $r \notin 3 \bbN_{0} + 2$.
\end{example}

Next we study the finite element spaces over a tetrahedron.
The proof follows similar ideas but the result is quite different:
only for finitely many polynomial degrees the bases are $\bbR$-invariant.

\begin{theorem} \label{theorem:recursiveconstruction:3D}
 Let $T$ be a tetrahedron, and let $r \in \bbN_{0}$. 
 Then the bases 
 \begin{gather*}
  \calA\calP_{r}\Lambda^{1}(T), \quad \calA\calP_{r}^{-}\Lambda^{1}(T),
   \quad 
   \calA\calP_{r}\Lambda^{2}(T), \quad \calA\calP_{r}^{-}\Lambda^{2}(T), 
   \\
   \calA\mathring\calP_{r}\Lambda^{1}(T), \quad \calA\mathring\calP_{r}^{-}\Lambda^{1}(T),
   \quad 
   \calA\mathring\calP_{r}\Lambda^{2}(T), \quad \calA\mathring\calP_{r}^{-}\Lambda^{2}(T)
 \end{gather*}
 are $\bbC$-invariant. Furthermore: 
 \begin{itemize}
\item The basis $\calA\calP_{r}    \Lambda^{1}(T)$ is $\bbR$-invariant if and only if          $r \in \{0,1,2,4,5,8\}$.   
  \item The basis $\calA\calP_{r}^{-}\Lambda^{1}(T)$ is $\bbR$-invariant if and only if          $r \in \{0,1,3,4,7\}$.        
  \item The basis $\calA\calP_{r}    \Lambda^{2}(T)$ is $\bbR$-invariant if and only if          $r \in \{0,1,2,4,5,8\}$.  
  \item The basis $\calA\calP_{r}^{-}\Lambda^{2}(T)$ is $\bbR$-invariant if and only if          $r \in \{0,1,2,3,4,6,7,10\}$.
  \item The basis $\calA\mathring\calP_{r}    \Lambda^{1}(T)$ is $\bbR$-invariant if and only if $r \in \{0,1,2,3,4,5,6,8,9,12\}$.
  \item The basis $\calA\mathring\calP_{r}^{-}\Lambda^{1}(T)$ is $\bbR$-invariant if and only if $r \in \{0,1,2,3,4,5,7,8,11\}$. 
  \item The basis $\calA\mathring\calP_{r}    \Lambda^{2}(T)$ is $\bbR$-invariant if and only if $r \in \{0,1,2,4,5,8\}$.  
  \item The basis $\calA\mathring\calP_{r}^{-}\Lambda^{2}(T)$ is $\bbR$-invariant if and only if $r \in \{0,1,2,3,4,6,7,10\}$.
 \end{itemize}
\end{theorem}

\begin{proof}
 Let $E_{0}, E_{1}, \dots, E_{5}$ be the edges of $T$
 and let $F_{0}, F_{1}, F_{2}, F_{3}$ be the faces of $T$.
 When $r \geq 1$, the geometric decomposition establishes 
 \begin{align*}
  \calA\calP_{r}\Lambda^{1}(T)
  &= 
  \calA\mathring\calP_{r}\Lambda^{1}(T)
  \cup 
  \bigcup_{i=0}^{3}
  \ext^{k,r}_{F_{i},T}
  \calA\mathring\calP_{r}\Lambda^{1}(F_{i})
  \cup 
  \bigcup_{i=0}^{5}
  \ext^{k,r}_{E_{i},T}
  \calA\mathring\calP_{r}\Lambda^{1}(E_{i})
  ,\\
\calA\calP_{r}^{-}\Lambda^{1}(T)
  &= 
  \calA\mathring\calP_{r}^{-}\Lambda^{1}(T)
  \cup 
  \bigcup_{i=0}^{3}
  \ext^{k,r,-}_{F_{i},T}
  \calA\mathring\calP_{r}^{-}\Lambda^{1}(F_{i})
  \cup 
  \bigcup_{i=0}^{5}
  \ext^{k,r,-}_{E_{i},T}
  \calA\mathring\calP_{r}^{-}\Lambda^{1}(E_{i})
  ,\\
\calA\calP_{r}\Lambda^{2}(T)
  &= 
  \calA\mathring\calP_{r}\Lambda^{2}(T)
  \cup 
  \bigcup_{i=0}^{3}
  \ext^{k,r}_{F_{i},T}
  \calA\mathring\calP_{r}\Lambda^{2}(F_{i})
  ,\\
\calA\calP_{r}^{-}\Lambda^{2}(T)
  &= 
  \calA\mathring\calP_{r}^{-}\Lambda^{2}(T)
  \cup 
  \bigcup_{i=0}^{3}
  \ext^{k,r,-}_{F_{i},T}
  \calA\mathring\calP_{r}^{-}\Lambda^{2}(F_{i})
  .
 \end{align*}
 By Theorem~\ref{theorem:recursiveconstruction:scalarvolume}, 
 the bases for $1$-forms over edges and $2$-forms over faces are $\bbR$-invariant and thus $\bbC$-invariant. 
 By Theorem~\ref{theorem:recursiveconstruction:2D}, 
 the bases for $1$-forms associated to the faces are $\bbC$-invariant.
 By the canonical isomorphisms, 
\begin{gather*}
  \calA\mathring\calP_{r+3}\Lambda^{1}(T) = \calJ_{2,r} \calA\calP_{r+1}^{-}\Lambda^{2}(T),
  \quad  
  \calA\mathring\calP_{r+3}^{-}\Lambda^{1}(T) = \calI_{2,r} \calA\calP_{r}\Lambda^{2}(T),
  \\
  \calA\mathring\calP_{r+2}\Lambda^{2}(T) = \calJ_{1,r} \calA\calP_{r+1}^{-}\Lambda^{1}(T),
  \quad  
  \calA\mathring\calP_{r+2}^{-}\Lambda^{2}(T) = \calI_{1,r} \calA\calP_{r}\Lambda^{1}(T).
 \end{gather*}
Moreover, 
 from Lemma~\ref{lemma:invariantconstant:1forms3D}
 and Lemma~\ref{lemma:invariantconstant:2forms3D}
 we know that the bases $\calA\calP_{0}\Lambda^{1}(T)$ and $\calA\calP_{0}\Lambda^{2}(T)$ are $\bbR$-invariant and thus $\bbC$-invariant.

 The recursive construction produces $\bbC$-invariant bases.
 This is seen with a short induction argument over the polynomial degree, analogous to preceding proofs. 
When $r=0$, then the bases $\calA\calP_{0}\Lambda^{1}(T)$ and $\calA\calP_{0}\Lambda^{2}(T)$ are $\bbR$-invariant 
 and the other bases are empty.
Suppose that $r > 0$ and that we have constructed $\bbC$-invariant bases for all polynomial degrees strictly less than $r$.
 Each of the bases $\calA\mathring\calP_{r}\Lambda^{1}(T)$, $\calA\mathring\calP_{r}^{-}\Lambda^{1}(T)$, $\calA\mathring\calP_{r}\Lambda^{2}(T)$ and $\calA\mathring\calP_{r}^{-}\Lambda^{2}(T)$ is either empty 
 or found via the canonical isomorphisms from a $\bbC$-invariant basis of lesser polynomial degree.
 Next, $\calA\calP_{r}\Lambda^{1}(T)$, $\calA\calP_{r}^{-}\Lambda^{1}(T)$, $\calA\calP_{r}\Lambda^{2}(T)$ and $\calA\calP_{r}^{-}\Lambda^{2}(T)$ 
 are defined via the geometric decomposition.
 Thus, the $\bbC$-invariance of the bases is clear.

 It remains to identify which of these bases are $\bbR$-invariant. 
 Each of these bases is $\bbR$-invariant
 if and only if 
 no stage of its recursive construction involves any non-$\bbR$-invariant basis associated to $1$-forms over a face.

 Therefore, we make use of Theorem~\ref{theorem:recursiveconstruction:2D}. 
 For any face $F$ of $T$, the bases 
 $\calA\mathring\calP_{r}\Lambda^{1}(F)$ and $\calA\mathring\calP_{r}^{-}\Lambda^{1}(F)$
 are not $\bbR$-invariant 
 if and only if 
 $r \in 3 \bbN_{0} + 3$ and $r \in 3 \bbN_{0} + 2$,
 respectively. 
 Put differently, 
 this is the case precisely for the bases 
 $\calA\mathring\calP_{3 j + 3}\Lambda^{1}(F)$ and $\calA\mathring\calP_{3 j + 2}^{-}\Lambda^{1}(F)$
 with $j \in \bbN_{0}$.
We follow which bases they enter throughout the recursive construction.

 To this end, we let $j \in \bbN_{0}$ be arbitrary.
We immediately see that the bases 
 $\calA\calP_{3 j + 3}\Lambda^{1}(T)$ and $\calA\calP_{3 j + 2}^{-}\Lambda^{1}(T)$
are not $\bbR$-invariant. 
 Applying the canonical isomorphisms, 
 we see that  
 $\calA\mathring\calP_{3 j + 5}^{-}\Lambda^{2}(T)$ and $\calA\mathring\calP_{3 j + 3}\Lambda^{2}(T)$
 are not $\bbR$-invariant. 
Thus,
 $\calA\calP_{3 j + 5}^{-}\Lambda^{2}(T)$ and $\calA\calP_{3 j + 3}\Lambda^{2}(T)$
 are not $\bbR$-invariant. 
Applying the canonical isomorphisms once more,
 we get that 
 $\calA\mathring\calP_{3 j + 7}\Lambda^{1}(T)$ and $\calA\mathring\calP^{-}_{3 j + 6}\Lambda^{1}(T)$
 are not $\bbR$-invariant. 
Iterating this argument, 
 we find that the following bases are not $\bbR$-invariant:
 \begin{align*}
  \calA\calP_{3 j + 3 + 4b }\Lambda^{1}(T), \quad \calA\calP_{3 j + 2 + 4b }^{-}\Lambda^{1}(T),
  \\
  \calA\mathring\calP_{3 j + 5 + 4b }^{-}\Lambda^{2}(T), \quad \calA\mathring\calP_{3 j + 3 + 4b }\Lambda^{2}(T),
  \\
  \calA\calP_{3 j + 5 + 4b }^{-}\Lambda^{2}(T), \quad \calA\calP_{3 j + 3 + 4b }\Lambda^{2}(T),
  \\
  \calA\mathring\calP_{3 j + 7 + 4b }\Lambda^{1}(T), \quad \calA\mathring\calP_{3 j + 6 + 4b }^{-}\Lambda^{1}(T),
\end{align*}
 where $j, b \in \bbN_{0}$. 
 However, all non-negative integers except $1$, $2$, and $5$ are in the set $3\bbN_{0} + 4\bbN_{0}$. 
Thus the theorem follows. 
\end{proof}

\begin{remark}
 Let us restate this result in the language of vector analysis:
 On any tetrahedron $T$ 
 and for any polynomial degree $r$, 
 the Raviart-Thomas space $\RT_{r}(T)$, the Brezzi-Douglas-Marini space $\BDM_{r}(T)$
 and the N\'ed\'elec spaces of the first kind $\Ned_{r}^{\rm fst}(T)$ and the second kind $\Ned_{r}^{\rm snd}(T)$
 have geometrically decomposed $\bbC$-invariant bases. 
 Moreover, 
 \begin{itemize}
  \item that basis of $\RT_{r}(T)$ is $\bbR$-invariant if $r \in \{ 0, 1, 2, 3, 4, 6, 7, 10 \}$,
  \item that basis of $\BDM_{r}(T)$ is $\bbR$-invariant if $r \in \{ 0, 1, 2, 4, 5, 8 \}$,
  \item that basis of $\Ned_{r}^{\rm fst}(T)$ is $\bbR$-invariant if $r \in \{ 0, 1, 3, 4, 7 \}$, 
  \item that basis of $\Ned_{r}^{\rm snd}(T)$ is $\bbR$-invariant if $r \in \{ 0, 1, 2, 4, 5, 8 \}$.
 \end{itemize}
\end{remark}

\begin{theorem} \label{theorem:recursiveconstruction:4D}
 Let $T$ be a $4$-simplex and let $r \in \bbN_{0}$. 
 Then the bases $\calA\calP_{r}\Lambda^{2}(T)$ and $\calA\calP_{r}^{-}\Lambda^{2}(T)$ are $\bbC$-invariant. 
\end{theorem}

\begin{proof}
 This follows the same line as the proofs of 
 Theorem~\ref{theorem:recursiveconstruction:2D} and Theorem~\ref{theorem:recursiveconstruction:3D}.
 The base case is addressed by Lemma~\ref{lemma:invariantconstant:2forms4D}. 
\end{proof}

\begin{example} \label{example:interestingproperties}
 We outline the $\bbC$-invariant basis $\calA\calP_{2}^{-}\Lambda^{1}(T)$ 
 and its geometric decomposition
 over a tetrahedron $T$. 
To an edge with vertex indices $i$ and $j$ we associate 
 the higher-degree Whitney forms $\lambda_{i}\whitney_{ij}$ and $\lambda_{j}\whitney_{ij}$.
 To each face $F$ with vertices indices $i,j,k$ we associate two forms  
 \begin{align*}
\lambda_{i} \whitney_{jk} - \xi_{3}^{ } \lambda_{j} \whitney_{ik} + \xi_{3}^{2} \lambda_{k} \whitney_{ij}
  ,
  \quad 
\lambda_{i} \whitney_{jk} - \xi_{3}^{2} \lambda_{j} \whitney_{ik} + \xi_{3}^{ } \lambda_{k} \whitney_{ij}
  ,
 \end{align*}
 where $\xi_{3} = \exp(2\iunit\pi/3 )$ as before. Finally, the forms
 \begin{align*}
2( \lambda_{0}\lambda_{1} \whitney_{23} + \lambda_{0}\lambda_{2} \whitney_{13} - \lambda_{1}\lambda_{3} \whitney_{02} - \lambda_{2}\lambda_{3} \whitney_{01} )
  ,
  \\
2( \lambda_{0}\lambda_{2} \whitney_{13} + \lambda_{0}\lambda_{3} \whitney_{12} + \lambda_{1}\lambda_{2} \whitney_{03} + \lambda_{1}\lambda_{3} \whitney_{02} )
  ,
  \\
2( \lambda_{0}\lambda_{3} \whitney_{12} - \lambda_{0}\lambda_{1} \whitney_{23} - \lambda_{1}\lambda_{2} \whitney_{03} - \lambda_{2}\lambda_{3} \whitney_{01} )
  .
 \end{align*}
 are associated with the tetrahedron itself.

We study the elementwise mass matrix in the special case where $T$ is a regular tetrahedron.  
 In contrast to scalar forms, some non-trivial orthogonality relations arise. 
One checks that the bases $\calA\calP_{0}\Lambda^{1}(T)$ and $\calA\calP_{0}\Lambda^{2}(T)$
 are each orthogonal with respect to the respective standard products. 
 One verifies the same for each of the bases
 $\calA\mathring\calP_{2}^{-}\Lambda^{2}(T)$ and $\calA\mathring\calP_{3}^{-}\Lambda^{1}(T)$,
 obtained along the respective canonical isomorphisms.
 Similarly,
 suppose that $F$ is a regular triangle, for example, a face of $T$. 
 Calculations show that the basis $\calA\calP_{0}\Lambda^{1}(F)$
 is orthogonal with respect to the standard products,
 as is $\calA\mathring\calP_{2}^{-}\Lambda^{1}(F)$ 
 and the extension $\ext^{1,2,-}_{F,T} \calA\mathring\calP_{2}^{-}\Lambda^{1}(F)$. 
While equilateral triangles and regular tetrahedra are obviously idealized special cases,
 these exemplary orthogonality relations suggest further computational studies. 
\end{example}

\begin{remark}
 Our search for invariant bases has led to positive results: 
 Theorem~\ref{theorem:recursiveconstruction:2D}, Theorem~\ref{theorem:recursiveconstruction:3D}, and Theorem~\ref{theorem:recursiveconstruction:4D} 
 explicitly construct $\bbC$-invariant geometrically decomposed bases. 
 We have pointed out conditions on the polynomial degree for the existence of bases
 that are $\bbR$-invariant, that is, invariant under permutation of indices up to sign change. 
 These conditions on $\bbR$-invariance are sufficient and we conjecture that they are necessary as well.

 If instead one of the finite element spaces above has an $\bbR$-invariant basis
 for a polynomial degree $r$ not listed above,
 then that basis must satisfy 
 constraints related to geometric decompositions in the sense of Section~\ref{sec:extension}. 
 Because of Theorem~\ref{theorem:geometricdecomposition:sullivan} and \ref{theorem:geometricdecomposition:whitney},
 any geometrically decomposed $\bbR$-invariant basis 
 induces $\bbR$-invariant bases on subsimplices and for finite element spaces with vanishing boundary traces,
 and we can then use a canonical isomorphism to get $\bbR$-invariant bases on spaces of lower degree. 
 
 For example, Djokovi\'c and Malzan's results already rule out $\bbR$-invariant bases 
 for several spaces of constant $k$-forms.
 Going further, 
 there is no $\bbR$-invariant basis for $\mathring\calP_{2}^{-}\Lambda^{1}(F)$
 or 
 geometrically decomposed $\bbR$-invariant basis for $\calP_{2}^{-}\Lambda^{1}(F)$
 over a triangle $F$.
A fortiori, 
 $\calP_{2}^{-}\Lambda^{1}(T)$ over a tetrahedron $T$ has no geometrically decomposed $\bbR$-invariant basis.
\end{remark}

\begin{remark} Throughout this section, we have emphasized the ``bottom-up'' recursive construction of higher-degree bases. 
 Conversely, given a basis for a higher-degree finite element space, 
 we can follow the recursion ``top-down'',   
 under the condition that the geometric decomposition can be used at every step. 
 The latter condition is not necessarily true for every basis. 
 In any case, no further recursion step is possible whenever the recursion leads to zeroth-degree finite element spaces: 
 there, no geometric decomposition exists. 
\end{remark}

\paragraph{Acknowledgement}
The author would like to thank Dragomir Djokovi\'c for helpful advice relating to Lemma~\ref{lemma:invariantconstant:2forms4D}, 
and the anonymous referees for numerous helpful remarks.
Moreover, the author acknowledges the hospitality of the University of Minnesota, where parts of this research was done. 
Helpful discussions with Douglas N.\ Arnold and Snorre H.\ Christiansen are gratefully appreciated.

\end{document}